\documentclass[leqno]{amsart}
\usepackage{amsmath}
\usepackage{amssymb}
\usepackage{amsthm}
\usepackage{enumerate}
\usepackage[mathscr]{eucal}
\theoremstyle{plain}
\usepackage{tikz}
\newtheorem{theorem}{Theorem}[section]
\newtheorem{lemma}[theorem]{Lemma}
\newtheorem{prop}[theorem]{Proposition}
\theoremstyle{definition}
\newtheorem{definition}[theorem]{Definition}
\newtheorem{remark}[theorem]{Remark}

\newtheorem{example}[theorem]{Example}

\newtheorem{cor}[theorem]{Corollary}
\theoremstyle{remark}
\begin{document}

\title [On uniform BPB type approximations of linear operators]{On uniform Bishop-Phelps-Bollob\'as type approximations of linear operators and preservation of geometric properties}

\author[Debmalya Sain, Arpita Mal, Kalidas Mandal and Kallol Paul]{Debmalya Sain, Arpita Mal, Kalidas Mandal and Kallol Paul}

\address[Sain]{Department of Mathematics\\ Indian Institute of Science\\ Bengaluru 560012\\ Karnataka\\ India}
\email{saindebmalya@gmail.com}

\address[Mal]{Department of Mathematics\\ Jadavpur University\\ Kolkata 700032\\ India}
\email{arpitamalju@gmail.com}

\address[Mandal]{Department of Mathematics\\ Jadavpur University\\ Kolkata 700032\\ India}
\email{kalidas.mandal14@gmail.com}

\address{(Paul)Department of Mathematics, Jadavpur University, Kolkata 700032, India}
\email{kalloldada@gmail.com}

%\thanks will become a 1st page footnote.
\thanks{Dr. Debmalya Sain feels elated to acknowledge the motivating presence of his beloved junior brothers, Dr. Chandrodoy Chattopadhyay and Mr. Krittish Roy, in his life! Miss Arpita Mal would like to thank UGC, Govt. of India for the financial support. Mr. Kalidas Mandal would like to thank CSIR, Govt. of India for the financial support. The research of Prof. Kallol Paul is supported by project MATRICS (Ref. no. MTR/2017/000059), SERB, DST, Govt. of India.}

%    General info
\subjclass[2010]{Primary 46B20, Secondary 47L05}
\keywords{Bishop-Phelps-Bollob\'as property, norm attainment, extreme contractions, smooth operators}


%    General info

\date{}
\maketitle
\begin{abstract}
We study uniform $\epsilon-$BPB approximations of bounded linear operators between Banach spaces from a geometric perspective. We show that for sufficiently small positive values of $\epsilon,$ many geometric properties like smoothness, norm attainment and extremality of operators are preserved under such approximations. We present examples of pairs of Banach spaces satisfying non-trivial norm preserving uniform $\epsilon-$BPB approximation property in the global sense. We also study these concepts in case of bounded linear operators between Hilbert spaces. Our approach in the present article leads to the improvement and generalization of some earlier results in this context.  
\end{abstract}

\section{Introduction}
The purpose of this article is to study a uniform version of Bishop-Phelps-Bollob\'as (BPB) type approximations in the space of bounded linear operators between Banach spaces and the corresponding preserver problems from a geometric viewpoint. We would like to begin with a short description of the motivation behind such an exploration. 

  The well known Bishop-Phelps Theorem \cite{BP} asserts that norm attaining functionals are dense in the dual of any Banach space.  A quantitative version of this profound result was presented by Bollob\'as in \cite{B}, and in the combined form, the Bishop-Phelps-Bollob\'as Theorem is one of the keystones of modern functional analysis and approximation theory beyond any shadow of doubt. Possible generalizations of the Bishop-Phelps-Bollob\'as Theorem, from the case of functionals to the case of operators, constitute an interesting area of research that remains active till date. We refer the readers to \cite{AAGM,ACGM,CGKS,D1,D2,DKL,DKKLM,DKLM,KL,Sa} for some of the prominent works in this context. Our present work draws its motivation from some of these, most notably \cite{AAGM,D2,DKKLM,DKLM,Sa}. However, although most of these works deal primarily with the analytic aspects of BPB type approximations of operators, our goal is to study the geometric properties of such approximations. More precisely, we study the preservation of important geometric properties, like smoothness and extreme contractions, under sufficiently accurate BPB type approximations of bounded linear operators. Let us now establish the notations and the terminologies that are essential to understand the main concepts related to BPB type approximations of bounded linear operators between Banach spaces.
  
  Letters $\mathbb{X}$, $\mathbb{Y}$ denote Banach spaces and the letter $\mathbb{H}$ is kept reserved for Hilbert spaces. Throughout the article, we assume that the underlying scalar field is $\mathbb{R}$, the field of real numbers. Let $B_\mathbb{X}=\{x\in \mathbb{X}:\|x\|\leq 1\}$  and $S_\mathbb{X}=\{x\in \mathbb{X}:\|x\|= 1\}$ denote the unit ball and the unit sphere of $\mathbb{X},$ respectively. For $x\in \mathbb{X}$ and $r>0,$ let $B(x,r)=\{z\in \mathbb{X}:\|x-z\|<r\}$ denote the open ball centered at $x$ with radius $r.$ Let $\mathbb{L}(\mathbb{X}, \mathbb{Y})$ $(\mathbb{K}(\mathbb{X}, \mathbb{Y}))$ denote the space of all bounded (compact) linear operators from $\mathbb{X}$ to $\mathbb{Y},$ endowed with the usual operator norm. Given Banach spaces $\mathbb{X}$ and $\mathbb{Y},$ Iso$(\mathbb{X}, \mathbb{Y})$ denotes the class (possibly empty) of all isometries from $\mathbb{X}$ to $\mathbb{Y}.$  For $T \in \mathbb{L}(\mathbb{X}, \mathbb{Y}),$ and for a real number $\delta > 0,$ the norm attainment set $M_T$ and the $\delta-$approximate norm attainment set $M_{T}(\delta)$ of $T$ are defined respectively as $M_T=\{x\in S_\mathbb{X}:\|Tx\|=\|T\|\}$ and $M_T(\delta)=\{x\in S_{\mathbb{X}}:\|Tx\|>\|T\|-\delta\}.$  Given a Banach space $\mathbb{X},$  $\mathbb{X}^*$ denotes the dual space of $\mathbb{X}.$ The famous Bishop-Phelps-Bollob\'as Theorem \cite{B} can now be presented in the following form: \\
  Let $\epsilon>0$ be arbitrary. If $x\in B_\mathbb{X}$ and $x^*\in S_{\mathbb{X}^*}$ are such that $|1-x^*(x)|<\frac{\epsilon^2}{4},$ then there are elements $y^*\in S_{\mathbb{X}^*}$ and $y\in M_{y^*}$  such that $\|y-x\|<\epsilon$ and $\|y^*-x^*\|<\epsilon.$
  
  The above result motivated Acosta et al. \cite{AAGM} to introduce the following definition in the study of BPB type approximations of bounded linear operators between Banach spaces. We note that the concept introduced in \cite{AAGM} is space dependent but independent of any individual operator between them. 
  
  \begin{definition}
  	Let $\mathbb{X},\mathbb{Y}$ be real or complex Banach spaces. We say that the couple $(\mathbb{X},\mathbb{Y})$ satisfies Bishop-Phelps-Bollob\'as property (BPBp in abbreviated form) for operators if given $\epsilon>0,$ there are $\eta(\epsilon)>0$ and $\beta(\epsilon)>0$ with $\lim_{t\to 0}\beta(t)=0$ such that for all $T\in S_{\mathbb{L}(\mathbb{X},\mathbb{Y})},$ if $x_0\in S_\mathbb{X}$ is such that $\|Tx_0\|>1-\eta(\epsilon),$ then there exists an operator $A\in S_{\mathbb{L}(\mathbb{X},\mathbb{Y})}$ and a point $u_0\in M_A$ that satisfy the following conditions:
  	  \[\|u_0-x_0\|<\beta(\epsilon)~\text{and~}\|A-T\|<\epsilon.\]
  \end{definition}
  
  A local version of the above concept, dependent on an individual operator between the given Banach spaces, was introduced in \cite{Sa}.

\begin{definition}\label{def-uniform}
	Let $\mathbb{X},\mathbb{Y}$ be Banach spaces and let $T\in S_{\mathbb{L}(\mathbb{X},\mathbb{Y})}.$ Let $\epsilon>0$ be fixed. We say that $A\in S_{\mathbb{L}(\mathbb{X},\mathbb{Y})}$ is a uniform $\epsilon-$BPB approximation of $T,$ if there exists $\delta(\epsilon)>0$ such that for each $x_0\in S_{\mathbb{X}}$ with $\|Tx_0\|>1-\delta(\epsilon)$ there exists $u_0\in M_A$ satisfying the following conditions:
	$$\|u_0-x_0\|<\epsilon~ \text{and}~\|T-A\|<\epsilon.$$
\end{definition}

It is easy to see that in the finite-dimensional case, every linear operator is a uniform $\epsilon-$BPB  approximation of itself, for every $\epsilon > 0.$ In view of this, for $T, A \in \mathbb{L}(\mathbb{X}, \mathbb{Y}),$ we say that $A$ is a \textit{non-trivial} uniform $\epsilon-$BPB approximation of $T,$ if in addition to the fact that $A$ is a uniform $\epsilon-$BPB approximation of $T,$ we also have that $A \neq T.$ We also recall the following global definition from \cite{Sa}:

\begin{definition}
	Let $\mathbb{X},\mathbb{Y}$ be Banach spaces. Let $\mathcal{F}$ be a family of norm one linear operators in $\mathbb{L}(\mathbb{X},\mathbb{Y}).$ We say that the pair $(\mathbb{X},\mathbb{Y})$ has uniform sBPBp with respect to $\mathcal{F}$ if given $\epsilon>0,$ there exists $\eta(\epsilon)>0$ such that whenever $T\in \mathcal{F}$ and $x_0\in S_{\mathbb{X}}$ are such that $\|Tx_0\|>1-\eta(\epsilon),$ there exists $x_1\in M_T$ such that $\|x_1-x_0\|<\epsilon.$
\end{definition}

It is apparent that the norm attainment set of an operator plays a central role in BPB type approximations in operator spaces. We study this connection in more detail by introducing the following definition:

\begin{definition}\label{def-pair}
	Let $\mathbb{X},\mathbb{Y}$ be Banach spaces. We say that the pair $(\mathbb{X},\mathbb{Y})$ has non-trivial norm preserving uniform BPB approximation property, if given any $T\in S_{\mathbb{L}(\mathbb{X},\mathbb{Y})}\setminus \text{Iso}(\mathbb{X},\mathbb{Y})$ and any $\epsilon>0,$ $T$ admits a non-trivial uniform $\epsilon-$BPB approximation $A_\epsilon$ such that $M_T=M_{A_\epsilon}.$
\end{definition}

The primary objective of this article is to illustrate that many important geometric concepts associated with bounded linear operators between Banach spaces are preserved under uniform $\epsilon-$BPB type approximations, when $\epsilon > 0$ is sufficiently small. It is clear that our approach is related to the preserver problems in connection with BPB type approximations in operator spaces, from a geometric point of view. After this introductory section, the paper is demarcated into four sections. In the second section, we collect some basic definitions related to the geometric properties of a Banach space, which are extensively used throughout the article. In the third section, we study uniform $\epsilon-$BPB approximations on $\mathbb{L}(\mathbb{X}, \mathbb{Y})$ with emphasis on the preservation of the norm attainment set of an operator. This allows us to study the geometry of the dual space of a given Banach space, from the perspective of BPB type approximations and in connection with geometric properties like smoothness and extreme contractions. We also improve an earlier result on uniform sBPBp obtained in \cite{Sa}. In the fourth section, we present some concrete operator theoretic results in the setting of $\ell_\infty^{n}, \ell_1^{n}$ and $\ell_p^2$ spaces, where $p \in \mathbb{N} \setminus \{1, 2\},$ that further illustrate the connection between BPB type approximations of linear operators and geometries of the underlying spaces. In particular, we show that the isometries on these spaces can be characterized by the nonexistence of non-trivial uniform $\epsilon-$BPB type approximations. In the fifth and the final section of this article, we consider the case of bounded linear operators between Hilbert spaces, in light of uniform $\epsilon-$BPB type approximations. In \cite{SP}, it has been proved that the norm attainment set of a bounded linear operator between Hilbert spaces is a special subset of the unit sphere of the domain space, in the sense that either it is empty or it is the unit sphere of some subspace of the domain space. This observation finds its use in the study of preservation of geometric properties of an operator between Hilbert spaces, under uniform $\epsilon-$BPB type approximations. We also obtain a characterization of Hilbert spaces among $\ell_p^n$ spaces, in terms of isometries and uniform $\epsilon-$BPB type approximations. Our results in this article illustrate that BPB type approximations of bounded linear operators between Banach (Hilbert) spaces are intimately related to the geometries of the underlying spaces and can be used as an effective tool in the study of the geometry of operator spaces.

\section{Preliminary}

For a non-zero vector $x\in \mathbb{X},$  $J(x)$ denotes the set of all supporting linear functionals of $x,$ i.e., $J(x)=\{f\in S_{\mathbb{X}^*}:f(x)=\|x\|\}.$ By the Hahn-Banach Theorem, it is clear that $J(x)\neq \emptyset.$ A non-zero vector $x$ is said to be smooth if $J(x)$ is singleton and $\mathbb{X}$ is said to be a smooth Banach space if every non-zero vector of $\mathbb{X}$ is smooth. 

$\mathbb{X}$ is said to be a strictly convex Banach space if for any $x,y\in S_\mathbb{X}$ with $x\neq y,$ $\|(1-t)x+ty\|<1$ holds for all $t\in (0,1).$ Equivalently, $\mathbb{X}$ is strictly convex if and only if $S_\mathbb{X}$ does not contain any non-trivial straight line segment.\\
A finite-dimensional Banach space $\mathbb{X}$ is said to be polyhedral if $B_{\mathbb{X}}$ contains only finitely many extreme points. Equivalently, $\mathbb{X}$ is a polyhedral Banach space, if $B_\mathbb{X}$ is a polyhedron.
Let us recall the following three definitions related to polyhedral Banach spaces. We note that the third definition was introduced in \cite{SPBB} in order to study the geometry of polyhedral Banach spaces in more detail.
\begin{definition}
	A polyhedron $P$ is a non-empty compact subset of $\mathbb{X}$ which is the intersection of finitely many closed half-spaces of $\mathbb{X},$ i.e., $P=\cap_{i=1}^rM_i,$ where $M_i$ are closed half-spaces in $\mathbb{X}$ and $r\in \mathbb{N}.$ The dimension $dim(P)$ of the polyhedron $P$ is defined as the dimension of the subspace generated by the differences $v-w$ of vectors $v,w\in P.$
\end{definition}

\begin{definition}
	A polyhedron $Q$ is said to be a face of the polyhedron $P$ if either $Q=P$ or if we can write $Q=P\cap \delta M,$ where $M$ is a closed half-space in $\mathbb{X}$ containing $P$ and $\delta M$ denotes the boundary of $M.$ If $dim(Q)=i,$ then $Q$ is called an $i-$face of $P$. $(n-1)-$faces of $P$ are called facets of $P$ and $1-$faces of $P$ are called edges of $P.$
\end{definition}

\begin{definition}
	Let $\mathbb{X}$ be a finite-dimensional polyhedral Banach space. Let $F$ be a facet of the unit ball $B_\mathbb{X}$ of $\mathbb{X}.$ A functional $f\in S_{\mathbb{X}^*}$ is said to be a supporting functional corresponding to the facet $F$ of the unit ball $B_\mathbb{X}$ if the following two conditions are satisfied:\\
	(i) $f$ attains norm at some point $v$ of $F,$\\
	(ii) $F=(v+\ker(f))\cap S_\mathbb{X}.$
\end{definition}
$\mathbb{X}$ is said to be Kadets-Klee if a sequence $\{x_n\}$ in $\mathbb{X}$ and $x\in \mathbb{X}$ is such that $x_n\stackrel{w}{\rightharpoonup} x$ and $\|x_n\|\to \|x\|,$ then $x_n\to x.$\\
$\mathbb{X}$ is said to be uniformly convex if for each $0<\epsilon\leq 2,$ $\delta_{\mathbb{X}}(\epsilon)>0$ holds, where the function $\delta_{\mathbb{X}}:[0,2]\to[0,1]$ is defined by: $$\delta_{\mathbb{X}}(\epsilon)=\inf\Big\{1-\frac{\|x+y\|}{2}:x,y\in S_\mathbb{X},\|x-y\|\geq \epsilon\Big\}.$$   

$\mathbb{X}$ is said to be uniformly smooth, if $\lim_{t\to 0^+}\frac{\rho_\mathbb{X}(t)}{t}=0$ holds, where the function $\rho_\mathbb{X}:(0,+\infty)\to[0,+\infty)$ is defined by:
$$\rho_{\mathbb{X}}(t)=\sup\Big\{\frac{\|x+ty\|+\|x-ty\|}{2}-1:x,y\in S_{\mathbb{X}}\Big\}.$$ 

For $x,y\in \mathbb{X},$ $x$ is said to Birkhoff-James orthogonal \cite{Bi,J} to $y$, written as $x\perp_B y$ if $\|x+\lambda y\|\geq \|x\|$ for all $\lambda \in \mathbb{R}.$ $x$ is said to be strongly Birkhoff-James orthogonal \cite{PSJ} to $y,$ written as $x\perp_{SB}y,$  if $\|x+\lambda y\|>\|x\|$ for all non-zero scalar $\lambda.$ \\
A linear projection $P$ on $\mathbb{X}$ is called an $L-$projection  if $\|x\|=\|Px\|+\|x-Px\|$ for all $x\in \mathbb{X}.$ A closed subspace $J \subseteq \mathbb{X}$ is called an $L-$summand \cite{HWW} if it is the range of an $L-$projection.\\
For a convex set $C,~\text{int}_r(C)$ denotes the relative interior of the set $C,$ i.e., $x\in \text{int}_r(C)$ if there exists $\epsilon>0$ such that $B(x,\epsilon)\cap \text{affine}(C)\subseteq C,$ where $\text{affine}(C)$ is the intersection of all affine sets containing $S$ and an affine set is defined as the translation of a vector subspace.  \\
An operator $T\in B_{\mathbb{L}(\mathbb{X},\mathbb{Y})}$ is said to be an extreme contraction if $T$ is an extreme point of the unit ball of $\mathbb{L}(\mathbb{X},\mathbb{Y}).$ A complete characterization of extreme contractions between Banach spaces in terms of geometric or analytic properties of the underlying spaces remains unknown. We refer the readers to \cite{G,L,SPM1,SRP,Sh} for more information in this regard.

\section{Uniform $\epsilon-$BPB approximation on $\mathbb{L}(\mathbb{X},\mathbb{Y})$}

We begin with a useful result illustrating that under sufficiently nice conditions, the approximate norm attainment set of a compact operator cannot be far from its norm attainment set.

\begin{theorem}\label{th-appnorm}
	Let $\mathbb{X}$ be a reflexive, Kadets-Klee Banach space and let $\mathbb{Y}$ be a Banach space. Let $T\in \mathbb{K}(\mathbb{X},\mathbb{Y}).$ Then for each $\epsilon>0,$ there exists $\delta(\epsilon)>0$ satisfying the following:
	$$\text{for~any~}0<\delta\leq\delta(\epsilon),~M_T(\delta)\subseteq \cup_{x\in M_T}B(x,\epsilon).$$
\end{theorem}
\begin{proof}
Let $\epsilon>0$ be given. If possible, suppose that there does not exist $\delta>0$ such that $M_T(\delta)\subseteq \cup_{x\in M_T}B(x,\epsilon).$ Then for each $n\in \mathbb{N},$ there exists $z_n\in M_T(\frac{1}{n})$ such that $z_n\notin \cup_{x\in M_T}B(x,\epsilon).$ Therefore, $\|Tz_n\|>\|T\|-\frac{1}{n}$ gives that $\|Tz_n\|\to \|T\|.$ Since $\mathbb{X}$ is reflexive, $B_{\mathbb{X}}$ is weakly compact. Therefore, $\{z_n\}$ has a weakly convergent subsequence in $B_{\mathbb{X}}.$ Without loss of generality, assume that $\{z_n\}$ weakly converges to $z\in B_{\mathbb{X}}.$ Now, since $T$ is a compact operator, $Tz_n\to Tz,$ i.e., $\|Tz_n\|\to \|Tz\|.$  Thus, $\|Tz\|=\|T\|.$ This clearly implies that $z\in M_T.$ Note that $z_n\stackrel{w}{\rightharpoonup} z$ and $\|z_n\|=\|z\|=1.$ Now, since $\mathbb{X}$ is Kadets-Klee, we get, $z_n\to z.$ This clearly contradicts that $z_n\notin \cup_{x\in M_T}B(x,\epsilon)$ for each $n\in \mathbb{N}.$ Therefore, there exists $\delta(\epsilon)>0$ such that $M_T(\delta(\epsilon))\subseteq \cup_{x\in M_T}B(x,\epsilon).$ Now, using \cite[Prop. 2.1]{Sa}, we have, for any $0<\delta\leq\delta(\epsilon),~M_T(\delta)\subseteq \cup_{x\in M_T}B(x,\epsilon).$ This completes the proof of the theorem.
\end{proof}

As a consequence of the above theorem, we get the following corollary, which improves on \cite[Th. 2.3]{Sa}. 

\begin{cor}
	Let $\mathbb{X}$ be a reflexive, Kadets-Klee Banach space and let $\mathbb{Y}$ be a Banach space. Let $T\in \mathbb{K}(\mathbb{X},\mathbb{Y})$ be a smooth operator. Then there exists $x_0\in S_\mathbb{X}$ such that for each $\epsilon>0,$ there exists $\delta(\epsilon)>0$ satisfying the following:
	$$\text{for~any~}0<\delta\leq\delta(\epsilon),~M_T(\delta)\subseteq B(x_0,\epsilon)\cup B(-x_0,\epsilon).$$
\end{cor}
\begin{proof}
	Since $T$ is smooth operator, from \cite[Th. 4.2]{PSG}, we get $M_T=\{\pm x_0\}$ for some $x_0\in S_{\mathbb{X}}.$ Now, the result follows from Theorem \ref{th-appnorm}.
\end{proof}

 We next discuss the preservation of the norm attainment set of an operator under uniform $\epsilon-$BPB approximations. We begin with the case of rank one operators between a reflexive Kadets-Klee Banach space and any Banach space having dimension greater than one.

\begin{theorem}\label{th-rank1}
	Let $\mathbb{X}$ be a reflexive, Kadets-Klee Banach space and $\mathbb{Y}$ be a Banach space such that $dim(\mathbb{Y})>1.$ Let $T\in S_{\mathbb{L}(\mathbb{X},\mathbb{Y})}$ be a rank one operator. Then for each $\epsilon>0,$ $T$ admits a non-trivial uniform $\epsilon-$BPB approximation $A_\epsilon$ such that $A_\epsilon$ is rank one and $M_{A_\epsilon}=M_T.$ 
\end{theorem}
\begin{proof}
	Since $T$ is rank one operator, there exists $f\in S_{\mathbb{X}^*}$ and $w\in S_{\mathbb{Y}}$ such that $Tx=f(x)w$ for all $x\in \mathbb{X}.$ Since  $dim(\mathbb{Y})>1,$ we can choose $u\in S_{\mathbb{Y}}$ such that $0<\|u-w\|<\frac{\epsilon}{2}.$ (Observe that if $dim(\mathbb{Y})=1,$ then we cannot find $u$ such that $0<\|u-w\|<\frac{\epsilon}{2},$ since $\mathbb{Y}$ is a real Banach space.) Using Theorem \ref{th-appnorm}, we say that there exists $\delta(\epsilon)>0$ such that $M_T(\delta(\epsilon))\subseteq\cup_{x\in M_T}B(x,\epsilon).$ Define $A_\epsilon\in \mathbb{L}(\mathbb{X},\mathbb{Y})$ by $A_\epsilon x=f(x)u$ for all $x\in \mathbb{X}.$ Then clearly, $\|A_\epsilon\|=1, T\neq A_\epsilon,$ $A_\epsilon$ is rank one and $M_T=M_{A_\epsilon}=M_f.$ We show that $A_\epsilon$ is a non-trivial uniform $\epsilon-$BPB approximation of $T.$ Let $z\in S_{\mathbb{X}}$ be such that $\|Tz\|>1-\delta(\epsilon).$ Then $z\in M_T(\delta(\epsilon)).$ Hence, there exists $x\in M_T$ such that $z\in B(x,\epsilon),$ i.e., $x\in M_{A_\epsilon}$ and $\|x-z\|<\epsilon.$ Now, for $z\in S_{\mathbb{X}},~\|(T-A_\epsilon)z\|=\|f(z)(w-u)\|\leq\|w-u\|<\frac{\epsilon}{2}$ implies that $\|T-A_\epsilon\|\leq \frac{\epsilon}{2}<\epsilon.$ This completes the proof of the theorem.
\end{proof}

Our next result shows that the property of being a non-extreme contraction for a compact operator on a reflexive Kadets-Klee Banach space can be preserved under uniform $\epsilon-$BPB approximations, along with the preservation of the norm attainment set.

\begin{theorem}\label{th-notextreme}
	Let $\mathbb{X}$ be a reflexive, Kadets-Klee Banach space and $\mathbb{Y}$ be a Banach space. Let $T\in S_{\mathbb{K}(\mathbb{X},\mathbb{Y})}$ be such that $T$ is not an extreme contraction. Then for each $\epsilon>0,$ $T$ admits a non-trivial uniform $\epsilon-$BPB approximation $A_\epsilon$ such that $M_{A_\epsilon}=M_T.$ Moreover, $A_\epsilon$ is not an extreme contraction. 
\end{theorem}
\begin{proof}
	Since $T$ is not an extreme contraction, there exist $T_1,T_2\in S_{\mathbb{L}(\mathbb{X},\mathbb{Y})}$ such that $T=\frac{1}{2}T_1+\frac{1}{2}T_2$ and $T\neq T_1,T\neq T_2.$ For each natural number $n>1,$ define $A_n=(1-\frac{1}{n})T+\frac{1}{n}T_1.$ Then $\|A_n\|\leq 1.$ Clearly, $T=\frac{n}{n+1}A_n+\frac{1}{n+1}T_2.$ Now, for each $x\in M_T,$
	\begin{eqnarray*}
	\|Tx\|&=& \|\frac{1}{n+1}T_2x+\frac{n}{n+1}A_nx\|\\
	\Rightarrow 1 &\leq &\frac{1}{n+1}\|T_2x\|+\frac{n}{n+1}\|A_nx\|\\
	&\leq &\frac{1}{n+1}+\frac{n}{n+1}\\
	&=&1.
	\end{eqnarray*}
This implies that $\|A_nx\|=1.$ Hence, $\|A_n\|=1$ and $M_T\subseteq M_{A_n}.$ Similarly, it can be shown that $M_{A_n}\subseteq M_T.$ Now, using Theorem \ref{th-appnorm}, we say that there exists $\delta(\epsilon)>0$ such that $M_T(\delta(\epsilon))\subseteq \cup_{x\in M_T}B(x,\epsilon).$ Clearly, $A_n\to T.$ So there exists $k\in \mathbb{N}$ such that $\|A_k-T\|<\epsilon.$ Let $z\in S_{\mathbb{X}}$ be such that $\|Tz\|>1-\delta(\epsilon).$ Then $z\in M_T(\delta(\epsilon))$ implies that there exists $x\in M_T$ such that $z\in B(x,\epsilon),$ i.e., $x\in M_{A_k}$ and $\|z-x\|<\epsilon.$ Thus, $A_k$ is a non-trivial uniform $\epsilon-$BPB approximation of $T$ such that $M_T=M_{A_k}.$ From the construction of $A_k,$ it is clear that $A_k$ is not an extreme contraction. This completes the proof of the theorem.
\end{proof}

The following theorem deals with extreme linear functionals defined on a finite-dimensional polyhedral Banach space. 

\begin{theorem}\label{th-polyextreme}
	Let $\mathbb{X}$ be a finite-dimensional polyhedral Banach space. Let $f\in S_{\mathbb{X}^*}$ be an extreme point of $B_{\mathbb{X}^*}$. Then there exists $\epsilon_0>0$ such that for each $0<\epsilon<\epsilon_0,$ $f$ is the only uniform $\epsilon-$BPB approximation of $f$.
\end{theorem}
\begin{proof}
	Since $f\in S_{\mathbb{X}^*}$ is an extreme point of $B_{\mathbb{X}^*},$ $f$ is a supporting linear functional corresponding to a facet, say $M$ of $B_{\mathbb{X}}.$ Choose $x\in \text{int}_r(M).$ Suppose 
	$$0<\epsilon_0<\min\{2,d(x,F):F~\text{is ~a~face~of~}B_\mathbb{X}, x\notin F\}.$$
	 Let $0<\epsilon<\epsilon_0.$ If possible, suppose that $f$ admits a non-trivial uniform $\epsilon-$BPB approximation $g_\epsilon$ and the corresponding constant is $\delta(\epsilon).$ Clearly, $x\in M_f.$ Hence, $|f(x)|>1-\delta(\epsilon).$ Therefore, there exists $z\in M_{g_\epsilon}$ such that $\|z-x\|<\epsilon<\epsilon_0.$ Now, there exists a face $F$ of $B_{\mathbb{X}}$ such that $z\in \text{int}_r(F).$ It is easy to observe that $F\subseteq M_{g_\epsilon}.$ If $x\notin F,$ then $\|x-z\|\geq d(x,F)>\epsilon_0,$ a contradiction. Therefore, $x\in F,$ i.e., $x\in M_{g_\epsilon}.$ Since $g_\epsilon$ attains its norm at a vector in $\text{int}_r(M),$ it is easy to observe that $M\subseteq M_{g_\epsilon}.$ Now, $|g_\epsilon(x)|=1$ and $\|f-g_\epsilon\|<\epsilon$ implies that $g_\epsilon(x)=1.$ Hence, $f$ and $g_\epsilon$ both supports the same facet of $B_{\mathbb{X}}.$ This gives that $f=g_\epsilon,$ which contradicts that $g_\epsilon$ is a non-trivial uniform $\epsilon-$BPB approximation of $f.$ This completes the proof of the theorem.   
\end{proof}

Combining above two theorems we can characterize extreme linear functionals on a finite-dimensional polyhedral Banach space in terms of uniform $\epsilon-$BPB approximations.

\begin{theorem}\label{th-funec}
		Let $\mathbb{X}$ be a finite-dimensional polyhedral Banach space. Let $f\in S_{\mathbb{X}^*}.$ Then the following are equivalent:\\
		(i) $f$ is an extreme contraction.\\
		(ii) There exists $\epsilon_0>0$ such that for each $0<\epsilon<\epsilon_0,$ $f$ does not admit a non-trivial uniform $\epsilon-$BPB approximation.\\
		(iii) For any $\epsilon>0,$ $f$ admits a uniform $\epsilon-$BPB approximation which is an extreme contraction.
\end{theorem}
\begin{proof}
	We note that the equivalence of $(i)$ and $(ii)$ follows from Theorem \ref{th-notextreme} and Theorem \ref{th-polyextreme}. Clearly $f$ is a uniform $\epsilon-$BPB approximation of itself by Definition \ref{def-uniform}, and therefore it follows that $(i)\Rightarrow (iii).$ We only prove that $(iii)\Rightarrow (i).$ Suppose $(iii)$ holds. Since $\mathbb{X}$ is a finite-dimensional polyhedral Banach space, the number of extreme contractions in $\mathbb{X}^*$ is finite. If possible, suppose that $f$ is not an extreme contraction. Choose
$$0<\epsilon< \min\{\|f-g\|:g~\text{is~an~extreme~contraction~in~}\mathbb{X}^*\}.$$
Now, for this $\epsilon,$ $f$ admits a uniform $\epsilon-$BPB approximation, say $g,$ which is an extreme contraction. Therefore, $\|f-g\|<\epsilon.$ This clearly contradicts the choice of $\epsilon.$ Therefore, we must have that $f$ is an extreme contraction. Thus, $(i)$ holds. This completes the proof of the theorem.
\end{proof}

 We invite the readers to compare and contrast the following theorem with Theorem \ref{th-polyextreme} and Theorem \ref{th-funec}. In particular, this shows that strict convexity plays an important role in the study of uniform $\epsilon-$BPB approximations.

\begin{theorem}\label{th-lin}
	Let $\mathbb{X}$ be a reflexive, Kadets-Klee, strictly convex Banach space such that $\mathbb{X}^*$ is a Kadets-Klee Banach space. Then for each $f\in S_{\mathbb{X}^*}$ and for each $\epsilon>0,$ $f$ admits a non-trivial uniform $\epsilon-$BPB approximation.
\end{theorem}
\begin{proof}
	Since $\mathbb{X}$ is a reflexive Banach space, $M_f\neq \emptyset.$ Using the strict convexity of $\mathbb{X},$ it is easy to observe that $M_f=\{\pm x\}$ for some $x\in S_{\mathbb{X}}.$ Without loss of generality, we may assume that $f$ is a supporting linear functional of $x,$ i.e., $f\in J(x).$ \\
	First suppose that $x$ is non-smooth. Then there exists $g\in J(x)$ such that $g\neq f.$ Since $J(x)$ is convex, $(1-t)f+tg\in J(x)$ for all $t\in [0,1].$ Clearly, there exists $t_\epsilon\in(0,1)$ such that $\|f-\{(1-t_\epsilon)f+t_\epsilon g\}\|<\epsilon.$ Let $g_\epsilon=(1-t_\epsilon)f+t_\epsilon g.$ Then $\|f-g_\epsilon\|<\epsilon$ and $M_{g_\epsilon}=\{\pm x\}.$ Now, by Theorem \ref{th-appnorm}, we have  $\delta(\epsilon)>0$ such that  $M_f(\delta(\epsilon))\subseteq B(x,\epsilon)\cup B(-x,\epsilon).$ Suppose $z\in S_\mathbb{X}$ satisfies $|f(z)|>1-\delta(\epsilon).$ Then $z\in M_f(\delta(\epsilon)).$ Now, either $z\in B(x,\epsilon)$ or $z\in B(-x,\epsilon).$ Moreover, $\pm x\in M_{g_\epsilon}.$ Thus, if $x$ is non-smooth, then $g_\epsilon$ is a non-trivial uniform $\epsilon-$BPB approximation of $f.$ \\
	Now, suppose that $x$ is smooth. Then $J(x)=\{f\}.$ Choose a sequence $\{x_n\}\subseteq S_{\mathbb{X}}$ such that $x_n\to x$ and $x_n\neq x$ for each $n.$ Suppose that $f_n$ is a supporting linear functional of $x_n$ for each $n.$ Since $\mathbb{X}^*$ is reflexive, $B_{\mathbb{X}^*}$ is weakly compact. Therefore, $\{f_n\}$ has a weakly convergent subsequence in $B_{\mathbb{X}^*}.$ Without loss of generality, we may assume that $\{f_n\}$ weakly converges to $g\in B_{\mathbb{X}^*}.$  Now,
	\begin{eqnarray*}
	|g(x)-1|=|g(x)-f_n(x_n)|&\leq &|g(x)-f_n(x)|+|f_n(x)-f_n(x_n)|\\
	&\leq & |g(x)-f_n(x)|+\|f_n\|\|x-x_n\|\\
	\Rightarrow |g(x)-1|&\leq &0, (~\text{taking~limit~}n\to\infty).
	\end{eqnarray*}
    Thus, $g(x)=1,$ i.e., $\|g\|=1$ and $g\in J(x).$ Therefore, $g=f$ and $f_n\stackrel{w}{\rightharpoonup}f.$ Clearly, $\|f_n\|\to \|f\|.$ Thus, we have $f_n\to f,$ since $\mathbb{X}^*$ is Kadets-Klee. Now, for $\epsilon>0,$ there exists $k\in \mathbb{N}$ such that $\|f_k-f\|<\epsilon$ and $\|x_k-x\|<\frac{\epsilon}{2}.$ Using Theorem \ref{th-appnorm}, we get $\delta(\epsilon)>0$ such that $M_f(\delta(\epsilon))\subseteq B(x,\frac{\epsilon}{2})\cup B(-x,\frac{\epsilon}{2}).$ Now, let $z\in S_{\mathbb{X}}$ such that $|f(z)|>1-\delta(\epsilon).$ Then $z\in M_f(\delta(\epsilon))$ implies that either $\|z-x\|<\frac{\epsilon}{2}$ or $\|z+x\|<\frac{\epsilon}{2}.$ Suppose that $\|z-x\|<\frac{\epsilon}{2}.$ Then $\|z-x_k\|\leq\|z-x\|+\|x-x_k\|<\epsilon.$ Similarly, $\|z+x\|<\frac{\epsilon}{2}$ gives that $\|z+x_k\|<\epsilon.$ However, $\pm x_k\in M_{f_k}$ and $\pm x_k\neq x$ implies that $f(x_k)\neq 1$ and $f(-x_k)\neq 1,$ i.e., $f_k\neq f.$  Thus, $f_k$ is a non-trivial uniform $\epsilon-$BPB approximation of $f,$ when $x$ is smooth. This completes the proof of the theorem.
\end{proof}

We record the following interesting observation in the form of a corollary to the above result.

\begin{cor}
	Let $\mathbb{X}$ be a uniformly smooth, Kadets-Klee, strictly convex Banach space. Then for each $f\in S_{\mathbb{X}^*}$ and for each $\epsilon>0,$ $f$ admits a non-trivial uniform $\epsilon-$BPB approximation.
\end{cor}
\begin{proof}
	Since $\mathbb{X}$ is uniformly smooth, $\mathbb{X}$ is reflexive and $\mathbb{X}^*$ is uniformly convex. Thus, $\mathbb{X}^*$ is a Kadets-Klee Banach space. Therefore, $\mathbb{X},$ $\mathbb{X}^*$ satisfies the hypothesis of Theorem \ref{th-lin} and hence the result follows from Theorem \ref{th-lin}. 
\end{proof}

The next theorem and its corollary show that the norm attainment set of an operator and that of its uniform $\epsilon-$BPB approximation cannot be far apart, under certain additional conditions.

\begin{theorem}\label{th-ball}
	Let $\mathbb{X},\mathbb{Y}$ be Banach spaces. Let $T\in S_{\mathbb{L}(\mathbb{X},\mathbb{Y})}$ be such that $M_T\neq \emptyset.$ Suppose $\sup\{\|Tz\|:z\in C\}<1$ holds for each closed subset $C$ of $S_{\mathbb{X}}$ with $d(M_T,C)>0.$ Then given any $\delta>0,$ there exists $\epsilon_0>0$ such that for $0<\epsilon<\epsilon_0$ if $A_{\epsilon}$ is a uniform $\epsilon-$BPB approximation of $T,$ then 
	$$M_{A_\epsilon}\subseteq \cup_{x\in M_T}B(x,\delta).$$
\end{theorem}
\begin{proof}
	Let $\delta>0$ be given. Suppose $C=S_\mathbb{X}\setminus \cup_{x\in M_T}B(x,\delta).$ Then clearly $d(M_T,C)>0.$ So by hypothesis, $\sup\{\|Tz\|:z\in C\}<1.$ There exists $\mu>0$ such that $\sup\{\|Tz\|:z\in C\}<1-\mu.$ Choose $0<\epsilon_0<\min\{\mu,\delta\}.$ Suppose $0<\epsilon<\epsilon_0$ and $A_{\epsilon}$ is a uniform $\epsilon-$BPB approximation of $T.$ Let $z\in M_{A_\epsilon}.$ Then 
	\begin{eqnarray*}
		|\|Tz\|-\|A_\epsilon z\||&\leq & \|Tz-A_\epsilon z\|\leq \|T-A_\epsilon\|<\epsilon\\
		\Rightarrow |\|Tz\|-1|&<&\epsilon<\mu\\
		\Rightarrow 1-\|Tz\|&<& \mu\\
		\Rightarrow 1-\mu&<& \|Tz\|.
	\end{eqnarray*}
	This clearly implies that $z\notin C,$ i.e., $z\in \cup_{x\in M_T}B(x,\delta).$ Thus, $M_{A_\epsilon}\subseteq \cup_{x\in M_T}B(x,\delta).$ This completes the proof of the theorem.   
\end{proof}

\begin{cor}\label{cor-maepsilon}
	Let $\mathbb{X},\mathbb{Y}$ be Banach spaces such that $dim(\mathbb{X})<\infty.$ Let $T\in S_{\mathbb{L}(\mathbb{X},\mathbb{Y})}.$ Then given any $\delta>0,$ there exists $\epsilon_0>0$ such that for $0<\epsilon<\epsilon_0$ if $A_{\epsilon}$ is a uniform $\epsilon-$BPB approximation of $T,$ then 
	$$M_{A_\epsilon}\subseteq \cup_{x\in M_T}B(x,\delta).$$
\end{cor}
\begin{proof}
	Since $dim(\mathbb{X})<\infty,~M_T\neq \emptyset.$ Suppose that $C$ is a closed subset of $S_{\mathbb{X}}$ such that $d(M_T,C)>0.$ Clearly, $C$ is a compact set since $S_{\mathbb{X}}$ is compact. Therefore, $\sup\{\|Tz\|:z\in C\}<1.$ Now, the result follows from Theorem \ref{th-ball}
\end{proof}

We now prove that exact preservation of the norm attainment set of an operator is possible under uniform $\epsilon-$BPB approximations, provided that the domain space is a finite-dimensional polyhedral Banach space.

\begin{theorem}
	Let $\mathbb{X}$ be a finite-dimensional polyhedral Banach space and let $\mathbb{Y}$ be a Banach space. Let $T\in S_{\mathbb{L}(\mathbb{X},\mathbb{Y})}.$ Then there exists $\epsilon_0>0$ such that for any $0<\epsilon<\epsilon_0,$ if $A_\epsilon$ is a uniform $\epsilon-$BPB approximation of $T,$ then $M_T=M_{A_\epsilon}.$ 	
\end{theorem}		
\begin{proof}
	$\mathbb{X}$ being finite-dimensional polyhedral Banach space, the number of all faces of $B_{\mathbb{X}}$ is finite. Let $\{F_i:1\leq i\leq n\}$ be the collection of all faces of $B_{\mathbb{X}}.$ For each $1\leq i\leq n,$ choose $x_i\in \text{int}_r(F_i).$ Choose
	\begin{eqnarray}\label{delta}
	0<\delta<\frac{1}{2}\min \{d(x_i,F_j):1\leq i,j\leq n, x_i\notin F_j\}.
	\end{eqnarray}
	Using Corollary \ref{cor-maepsilon}, we say that for this $\delta>0,$ there exists $\mu>0$ such that for each $0<\epsilon<\mu,$ if $A_\epsilon$ if a uniform $\epsilon-$BPB approximation of $T,$ then the following holds:
	\begin{eqnarray}\label{eq2}
	M_{A_\epsilon}\subseteq \cup_{z\in M_T}B(z,\delta).
	\end{eqnarray}
	Choose $0<\epsilon_0<\min\{\mu,\delta\}.$ Let $0<\epsilon<\epsilon_0.$ Suppose that $A_\epsilon$ is a uniform $\epsilon-$BPB approximation of $T$ and $\delta(\epsilon)$ is the corresponding constant. We show that $M_T=M_{A_\epsilon}$ in the following two steps.\\
	
	\textbf{Step 1:} $M_T\subseteq M_{A_\epsilon}.$\\
	Let $x\in M_T.$ Then there exists a face $F_i$ of $B_{\mathbb{X}}$ such that $x\in F_i\subseteq M_T.$ Clearly, $x_i\in M_T.$ So $\|Tx_i\|>1-\delta(\epsilon).$  Hence, there exists $z\in M_{A_\epsilon}$ such that $\|z-x_i\|<\epsilon.$ Now, there exists a face $F_j$ of $B_\mathbb{X}$ such that $z\in \text{int}_r(F_j).$ Since $z\in M_{A_\epsilon},$ we have $F_j\subseteq M_{A_\epsilon}.$ If $x_i\notin F_j,$ then from (\ref{delta}) we get $\|x_i-z\|\geq d(x_i,F_j)>2\delta>2\epsilon,$ a contradiction. Therefore, $x_i\in F_j,$ i.e., $x_i\in M_{A_\epsilon}.$ Moreover, $x_i\in \text{int}_r(F_i).$ Thus, $F_i\subseteq M_{A_\epsilon},$ i.e., $x\in M_{A_\epsilon}.$ This completes the proof of Step 1.\\
	
	\textbf{Step 2:} $M_{A_\epsilon}\subseteq M_T.$\\
	Let $x\in M_{A_\epsilon}.$ Then there exists a face $F_i$ of $B_{\mathbb{X}}$ such that $x\in F_i\subseteq M_{A_\epsilon}.$ Clearly, $x_i\in M_{A_\epsilon}.$ Using (\ref{eq2}), we get $z\in M_T$ such that $x_i\in B(z,\delta).$ Now, there exists a face $F_j$ of $B_\mathbb{X}$ such that $z\in \text{int}_r(F_j).$ Then clearly, $F_j\subseteq M_T.$ If $x_i\notin F_j,$ then from (\ref{delta}) we have $\|x_i-z\|\geq d(x_i,F_j)>2\delta,$ a contradiction. Therefore, $x_i\in F_j,$ i.e., $x_i\in M_T.$ Moreover, $x_i\in \text{int}_r(F_i).$ Thus, $F_i\subseteq M_T,$ i.e., $x\in M_T.$ This completes the proof of Step 2.
	
	Combining Step 1 and Step 2, we get $M_T=M_{A_\epsilon}.$ This completes the proof of the theorem.
\end{proof}	

We make note of the following proposition which should be compared with Theorem \ref{th-notextreme}.

\begin{prop}\label{th-linextreme}
	Let $\mathbb{X},\mathbb{Y}$ be Banach spaces such that there exist only finitely many extreme contractions in $\mathbb{L}(\mathbb{X},\mathbb{Y}).$ Let $T\in S_{\mathbb{L}(\mathbb{X},\mathbb{Y})}$ be such that $T$ is not an extreme contraction. Then there exists $\epsilon_0>0$ such that for any $0<\epsilon<\epsilon_0,$ if $A_\epsilon$ is a uniform $\epsilon-$BPB approximation of $T,$ then $A_\epsilon$ is not an extreme contraction. 
\end{prop}
\begin{proof}
	Suppose $\epsilon_0=\inf\{\|T-A\|:A\text{~is~an~extreme~contraction~in~}\mathbb{L}(\mathbb{X},\mathbb{Y})\}.$ Since the number of extreme contractions in $\mathbb{L}(\mathbb{X},\mathbb{Y})$ is finite, $\epsilon_0>0.$ Suppose $0<\epsilon<\epsilon_0$ and  $A_\epsilon$ is a uniform $\epsilon-$BPB approximation of $T.$ Then $\|T-A_{\epsilon}\|<\epsilon<\epsilon_0$ clearly implies that $A_{\epsilon}$ is not an extreme contraction. This completes the proof of the proposition.
\end{proof}

\begin{remark}\label{remark-nex}
	Let $\mathbb{X},\mathbb{Y}$ be finite-dimensional polyhedral Banach spaces. Let $T\in S_{\mathbb{L}(\mathbb{X},\mathbb{Y})}$ be such that $T$ is not an extreme contraction. Then there exists $\epsilon_0>0$ such that for any $0<\epsilon<\epsilon_0,$ if $A_\epsilon$ is a uniform $\epsilon-$BPB approximation of $T,$ then $A_\epsilon$ is not an extreme contraction. 
\end{remark}

We now introduce a property of a Banach space that is useful in studying uniform $\epsilon-$BPB approximations of operators.

\begin{definition}
	We say that a Banach space $\mathbb{X}$ satisfies Property $(P),$ if there exists $r_0>0$ such that given any $A\in S_{\mathbb{L}(\mathbb{X},\mathbb{X})}\setminus \text{Iso}(\mathbb{X},\mathbb{X}),$ there exists $x_A\in S_\mathbb{X}$ such that $B(x_A,r_0)\cap M_A=\emptyset.$
\end{definition}

We next demonstrate the importance of Property (P) in studying uniform $\epsilon-$BPB approximations of operators.

\begin{theorem}\label{th-propertyP}
	Let $\mathbb{X}$ be a Banach space such that $\text{Iso}(\mathbb{X},\mathbb{X})$ is a finite set and $\mathbb{X}$ satisfies Property $(P).$ Let $T\in \text{Iso}(\mathbb{X},\mathbb{X}).$  Then there exists $\epsilon_0>0$ such that for each $0<\epsilon<\epsilon_0,$ $T$ is the only uniform $\epsilon-$BPB approximation of $T.$ 
\end{theorem}
\begin{proof}
	Let $c=\inf\{\|T-V\|:V\in\text{Iso}(\mathbb{X},\mathbb{X}), T\neq V\}.$ Then clearly, $c>0.$ Since $\mathbb{X}$ satisfies Property $(P),$ there exists $r_0>0$ such that given any $A\in S_{\mathbb{L}(\mathbb{X},\mathbb{X})}\setminus \text{Iso}(\mathbb{X},\mathbb{X}),$ there exists $x_A\in S_\mathbb{X}$ such that $B(x_A,r_0)\cap M_A=\emptyset.$ Choose $0<\epsilon_0<\min\{c,r_0\}.$ Let $0<\epsilon<\epsilon_0.$ Suppose that $A_\epsilon$ is a uniform $\epsilon-$BPB approximation of $T$ and the corresponding constant is $\delta(\epsilon).$ If possible, let $A_\epsilon\in S_{\mathbb{L}(\mathbb{X},\mathbb{X})}\setminus \text{Iso}(\mathbb{X},\mathbb{X}).$ Then there exists $x_{A_\epsilon}=x,$ (say) such that $B(x,r_0)\cap M_{A_\epsilon}=\emptyset.$ Since $T$ is an isometry, $x\in M_T,$ i.e., $\|Tx\|>1-\delta(\epsilon).$ Therefore, there exists $y\in M_{A_\epsilon}$ such that $\|y-x\|<\epsilon<r_0.$ This implies that $y\in B(x,r_0)\cap M_{A_\epsilon},$ which is a contradiction. Hence, $A_\epsilon\in \text{Iso}(\mathbb{X},\mathbb{X}).$ Now, $\|T-A_\epsilon\|<\epsilon<c$ clearly implies that $T=A_\epsilon.$ Thus, $T$ is the only uniform $\epsilon-$BPB approximation of $T.$ This completes the proof of the theorem.
\end{proof}

Our next result proves the existence of a large family of Banach spaces having Property (P). Later on, in due course of time, we will deal with other types of spaces also having Property (P).

\begin{theorem}\label{th-polyP}
	Let $\mathbb{X}$ be a finite-dimensional polyhedral Banach space. Then $\mathbb{X}$ satisfies Property $(P).$
\end{theorem}
\begin{proof}
	Clearly, $B_{\mathbb{X}}$ has finitely many facets say, $M_1,M_2,\ldots,M_k.$ Choose $x_i\in \text{int}_r(M_i)$ for $1\leq i\leq k.$ Choose $$0<r_0<\min\{d(x_i,F):F \text{~is~a~face~of~} B_{\mathbb{X}},x_i\notin F, 1\leq i\leq k\}.$$
	Let $A\in S_{\mathbb{L}(\mathbb{X},\mathbb{X})}\setminus \text{Iso}(\mathbb{X},\mathbb{X}).$ If $x_i\in M_A$ for some $1\leq i\leq k,$ then it is easy to observe that $M_i\subseteq M_A.$ Since $A$ is not an isometry, clearly $x_i\notin M_A$ for some $1\leq  i\leq k.$ We claim that $B(x_i,r_0)\cap M_A=\emptyset.$  If possible, let $y\in B(x_i,r_0)\cap M_A.$ Then $y\in \text{int}_r(F)$ for some face $F$ of $B_{\mathbb{X}}.$ Now, $y\in M_A$ gives that $F\subseteq M_A.$  If $x_i\notin F,$ then $\|x_i-y\|\geq d(x_i,F)>r_0,$ a contradiction. Therefore, $x_i\in F.$ Thus, we get  $x_i\in M_A,$ which is again a contradiction. Hence, $B(x_i,r_0)\cap M_A=\emptyset.$ This completes the proof of the theorem.
\end{proof}

The above two theorems immediately yield the following corollary.

\begin{cor}\label{cor-isometry}
	Let $\mathbb{X}$ be a finite-dimensional polyhedral Banach space. Let $T\in \mathbb{L}(\mathbb{X},\mathbb{X})$ be an isometry. Then there exists $\epsilon_0>0$ such that for each $0<\epsilon<\epsilon_0,$ $T$ is the only uniform $\epsilon-$BPB approximation of $T.$ 
\end{cor}
\begin{proof}
	Suppose $U$ is an arbitrary isometry on $\mathbb{X}.$ Then $U$ maps extreme points of $B_{\mathbb{X}}$ to extreme points of $B_{\mathbb{X}}.$ Since $\mathbb{X}$ is a finite-dimensional polyhedral Banach space, $B_{\mathbb{X}}$ has only finitely many extreme points. Thus, there are only finitely many distinct isometries on $\mathbb{X}.$ Now, the result follows from Theorem \ref{th-polyP} and Theorem \ref{th-propertyP}.
\end{proof}

Our next two results further illustrate the role of the norm attainment set of an operator in the study of uniform $\epsilon-$BPB approximations.

\begin{prop}
	Let $\mathbb{X},\mathbb{Y}$ be Banach spaces. Suppose $T\in S_{\mathbb{L}(\mathbb{X},\mathbb{Y})}$ is such that $M_T$ is a non-empty discrete set. Then there exists $\epsilon_0>0$ such that for $0<\epsilon<\epsilon_0$ if $A_\epsilon$ is a non-trivial uniform $\epsilon-$BPB approximation of $T,$ then $|M_{A_\epsilon}|\geq |M_T|.$ 
\end{prop}
\begin{proof}
	Since $M_T$ is a closed discrete set, there exists a real number $c>0$ such that $\inf\{\|x-y\|:x,y\in M_T,x\neq y\}=c.$ Choose $0<\epsilon<\frac{c}{2}.$ Then $B(x,\epsilon)\cap B(y,\epsilon)=\emptyset,$ if $x,y\in M_T$ and $x\neq y.$ Let $A_\epsilon$ be a non-trivial uniform $\epsilon-$BPB approximation of $T$ and the corresponding constant is $\delta(\epsilon).$ Then for each $x\in M_T,~\|Tx\|>1-\delta(\epsilon).$ Therefore, there exits $z\in M_{A_\epsilon}$ such that $\|z-x\|<\epsilon.$ Thus, for $x,y\in M_T$ and $x\neq y,$ there exists $z_x,z_y\in M_{A_\epsilon}$ such that $z_x\in B(x,\epsilon)$ and $z_y\in B(y,\epsilon).$ Now, $B(x,\epsilon)\cap B(y,\epsilon)=\emptyset$ implies that $z_x\neq z_y.$ Therefore, $|M_T|\leq |M_{A_\epsilon}|.$
\end{proof}

\begin{prop}\label{prop-dist}
	Let $\mathbb{X},\mathbb{Y}$ be Banach spaces. Let $T\in S_{\mathbb{L}(\mathbb{X},\mathbb{Y})}$ be such that $M_T\neq \emptyset.$ Suppose that for each $\epsilon>0,$ $T$ admits a non-trivial uniform $\epsilon-$BPB approximation $A_{\epsilon}$ such that $M_{A_\epsilon}=M_T$. If $Z$ is a subspace of $\mathbb{X}$ with $d(M_T,Z)>0,$ then $\|T\|_Z<1.$
\end{prop}
\begin{proof}
	If possible, suppose that $\|T\|_Z=1.$ Then there exists a sequence $\{x_n\}\subseteq S_Z$ such that $\|Tx_n\|\to 1.$ Choose $0<\epsilon<d(M_T,Z).$ Let $\delta(\epsilon)$ be the constant of non-trivial uniform $\epsilon-$BPB approximation $A_\epsilon$ of $T$ such that $M_{A_\epsilon}=M_T.$ Choose a natural number $k$ such that $\|Tx_k\|>1-\delta(\epsilon).$ Then there exists $u_k\in M_{A_\epsilon}$ such that $\|u_k-x_k\|<\epsilon.$ Now, $d(M_T,Z)=d(M_{A_\epsilon},Z)\leq \|u_k-x_k\|<\epsilon.$ This contradicts the choice of $\epsilon.$ Therefore, $\|T\|_Z<1.$ This completes the proof of the proposition. 
\end{proof}

In the next two theorems, we further study the preservation of the operator norm attainment set under uniform $\epsilon-$BPB approximations, provided certain additional conditions are satisfied.

\begin{theorem}
	Let $\mathbb{X}$ be a reflexive, Kadets-Klee Banach space and let $\mathbb{Y}$ be a Banach space. Let $T\in \mathbb{K}(\mathbb{X},\mathbb{Y})$ be a norm one operator. Let $span \{M_T\}=X_1.$ If there exists a subspace $X_2$ of $\mathbb{X}$ such that $\mathbb{X}=X_1\oplus X_2,$ $\|T\|_{X_2}>0$ and $X_1\perp_B X_2,$ then for every $\epsilon>0,$ $T$ admits a non-trivial uniform $\epsilon-$BPB approximation $A_\epsilon$ such that $M_{A_\epsilon}=M_T.$ 
\end{theorem}
\begin{proof}
	Since $\mathbb{X}=X_1\oplus X_2,$ for each $x\in \mathbb{X},$ there exists $x_1\in X_1$ and $x_2\in X_2$ such that $x=x_1+x_2.$ For each $n\in \mathbb{N},$ define $A_n:\mathbb{X}\to \mathbb{Y}$ as follows:
	\[A_n(x_1+x_2)=Tx_1+(1-\frac{1}{n})Tx_2,~\text{where~}x_1\in X_1,x_2\in X_2.\]
	Now, 
	\begin{eqnarray*}
	\|A_n(x_1+x_2)\|&=&\|Tx_1+(1-\frac{1}{n})Tx_2\|\\
	                &=&\|(1-\frac{1}{n})(Tx_1+Tx_2)+\frac{1}{n}Tx_1\|\\
	                &\leq& (1-\frac{1}{n})\|T(x_1+x_2)\|+\frac{1}{n}\|Tx_1\|\\
	                &\leq& (1-\frac{1}{n})\|T(x_1+x_2)\|+\frac{1}{n}\|T\|\|x_1\|\\
	                &\leq& (1-\frac{1}{n})\|T(x_1+x_2)\|+\frac{1}{n}\|x_1+x_2\|,~(\text{since}~x_1\perp_Bx_2)\\
	                &\leq& (1-\frac{1}{n})\|x_1+x_2\|+\frac{1}{n}\|x_1+x_2\|\\
	                &=&\|x_1+x_2\|.
	\end{eqnarray*}
    Thus, $\|A_n\|\leq 1.$ If $x_1\in M_T,$ then $\|A_nx_1\|=\|Tx_1\|=1.$ So $\|A_n\|=1$ and $M_T\subseteq M_{A_n}.$ For $x_2\in S_{X_2},~\|A_nx_2\|=(1-\frac{1}{n})\|Tx_2\|<1$ and for $x_1\in S_{X_1}\setminus M_T,~\|A_nx_1\|=\|Tx_1\|<1.$ Thus, $M_T=M_{A_n}.$ Now, $\|(T-A_n)(x_1+x_2)\|=\frac{1}{n}\|Tx_2\|\leq \frac{1}{n}\|x_2\|\leq\frac{2}{n}\|x_1+x_2\|.$ Hence, $\|T-A_n\|\leq \frac{2}{n}$ gives that $T-A_n\to 0.$ Moreover, $\|T\|_{X_2}>0$ gives that there exists $x_2\in X_2$ such that $Tx_2\neq 0.$ Therefore, $A_n\neq T$ for each $n>1.$ For $\epsilon>0,$ there exists $k\in \mathbb{N}\setminus\{1\}$ such that $\|T-A_k\|<\epsilon.$ Since $\mathbb{X}$ is reflexive, Kadets-Klee Banach space and $T$ is compact operator, by Theorem \ref{th-appnorm}, we say that there exists $\delta(\epsilon)>0$ such that $M_T(\delta(\epsilon))\subseteq\cup_{x\in M_T}B(x,\epsilon).$ Now, proceeding similarly as in Theorem \ref{th-notextreme}, it can be shown that $A_k$ is a non-trivial uniform $\epsilon-$BPB approximation of $T.$ This completes the proof of the theorem.
\end{proof}

\begin{theorem}
	Let $\mathbb{X},\mathbb{Y}$ be Banach spaces. Let $T\in S_{\mathbb{L}(\mathbb{X},\mathbb{Y})}$ be such that $M_T=\{\pm x_0\}.$ Suppose $span\{x_0\}$ is an L-summand of $\mathbb{X}$ and $P$ is the L-projection with range$(P)=span\{x_0\}.$ Suppose $0<\|T\|_{\ker(P)}<1.$ Then for each $\epsilon>0,$ $T$ admits a non-trivial uniform $\epsilon-$BPB approximation $A_\epsilon$ such that $M_{A_\epsilon}=M_T.$  
\end{theorem}
\begin{proof}
	We first claim that for $\epsilon>0,$ there exists $\delta(\epsilon)>0$ such that $M_T(\delta(\epsilon))\subseteq B(x_0,\epsilon)\cup B(-x_0,\epsilon).$ If it is not true, then for each $n\in \mathbb{N},$ there exists $x_n\in M_T(\frac{1}{n})$ such that $x_n\notin B(x_0,\epsilon)\cup B(-x_0,\epsilon).$ Clearly, $\|Tx_n\|>1-\frac{1}{n}$ gives that $\|Tx_n\|\to 1.$ Now, proceeding similarly, as in the proof of \cite[Th. 3.1]{MPRS}, we can conclude that $\{x_n\}$ has a convergent subsequence converging to $ax_0,$ where $|a|=1.$ This contradicts that $x_n\notin B(x_0,\epsilon)\cup B(-x_0,\epsilon)$ for each $n.$ This proves our claim. Observe that each $z\in \mathbb{X}$ can be written as $z=\alpha x_0+h$ for some $\alpha \in \mathbb{R}$ and $h\in \ker(P).$ Now, for each $n\in \mathbb{N}$ define $A_n:\mathbb{X}\to\mathbb{Y}$ as follows:
	$$A_n(\alpha x_0+h)=\alpha Tx_0+(1-\frac{1}{n})Th, ~\text{where}~\alpha \in \mathbb{R},h\in \ker(P).$$
	It is easy to observe that $\|A_n\|=1,~M_{A_n}=M_T=\{\pm x_0\}$ and $A_n\to T.$ Now, $\|T\|_{\ker(P)}>0$ gives that there exists $h\in \ker(P)$ such that $Tx\neq 0$ and so $A_nx\neq Tx,$ i.e., $A_n\neq T$ for all $n>1.$ Now, proceeding similarly as in the proof of Theorem \ref{th-notextreme}, it can be shown that $T$ admits a non-trivial uniform $\epsilon-$BPB approximation $A_\epsilon$ of $T$ such that $M_{A_\epsilon}=M_T.$ 
\end{proof}

The norm attainment set of an operator is a key factor in determining its smoothness. In this spirit, we next present the following three results highlighting the connection between smoothness of an operator and uniform $\epsilon-$BPB approximations.

\begin{theorem}\label{th-sm}
	Let $\mathbb{X}$ be a Banach space and let $\mathbb{Y}$ be a smooth Banach space. Let $T\in S_{\mathbb{L}(\mathbb{X},\mathbb{Y})}$ be such that $M_T\neq \emptyset.$  Suppose that for each $\epsilon>0,$ $T$ admits a non-trivial uniform $\epsilon-$BPB approximation which is smooth. Then $T$ is smooth. 
\end{theorem}
\begin{proof}
Since $M_T\neq \emptyset,$ let $\pm x_0\in M_T.$ Following the same method as in \cite[Prop. 2.7]{Sa}, it can be shown that $M_T=\{\pm x_0\}.$ We claim that if $\{x_n\}$ is any norming sequence for $T$ then $\{x_n\}$ has a convergent subsequence converging to $ax_0,$ where $|a|=1.$ Let for each $k\in \mathbb{N},$ $\delta(\frac{1}{k})$ be the constant of non-trivial uniform $\frac{1}{k}-$BPB approximation say $A_k$, corresponding to the value $\epsilon=\frac{1}{k}.$ Clearly, $M_{A_k}\neq \emptyset.$ By hypothesis, $A_k$ is smooth and so by \cite[Th. 3.3]{SPMR}, $M_{A_k}=\{\pm y_k\},$ say. Now, for each $k\in \mathbb{N},$ there exists $x_{n_k}$ such that $\|Tx_{n_k}\|>1-\delta(\frac{1}{k}).$ Then either $\|y_k-x_{n_k}\|<\frac{1}{k}$ or $\|y_k+x_{n_k}\|<\frac{1}{k}.$ Also $\|Tx_0\|>1-\delta(\frac{1}{k})$ gives that either $\|y_k-x_0\|<\frac{1}{k}$ or $\|y_k+x_0\|<\frac{1}{k}.$ First suppose that $\|y_k-x_{n_k}\|<\frac{1}{k}$ and $\|y_k-x_0\|<\frac{1}{k}.$ Then $\|x_{n_k}-x_0\|\leq \|x_{n_k}-y_k\|+\|y_k-x_0\|<\frac{2}{k}.$ Also if we suppose that $\|y_k-x_{n_k}\|<\frac{1}{k}$ and $\|y_k+x_0\|<\frac{1}{k},$ then we have, $\|x_{n_k}+x_0\|\leq \|x_{n_k}-y_k\|+\|y_k+x_0\|<\frac{2}{k}.$ Similarly, other cases imply that either $\|x_{n_k}-x_0\|<\frac{2}{k}$ or $\|x_{n_k}+x_0\|<\frac{2}{k}.$ Thus, $\{x_{n_k}\}$ has a subsequence converging to $a x_0,$ where $|a|=1.$ This proves our claim. Now, by \cite[Th. 3.4]{SPMR}	we can conclude that $T$ is smooth.	
\end{proof}

\begin{theorem}\label{th-smooth}
	Let $\mathbb{X}$ be a reflexive, Kadets-Klee Banach space and let $\mathbb{Y}$ be a smooth Banach space such that $dim(\mathbb{Y})>1$. Let $T\in \mathbb{K}(\mathbb{X},\mathbb{Y})$ be such that $\|T\|=1.$ Then $T$ is smooth if and only if for every $\epsilon>0,$ $T$ admits a non-trivial uniform $\epsilon-$BPB approximation which is also smooth. 
\end{theorem}
\begin{proof}
	The sufficient part follows from Theorem \ref{th-sm}. We only prove the necessary part of the theorem. If $rank(T)=1,$ then the result follows from Theorem \ref{th-rank1}. Let $rank(T)>1.$ Since $T$ is smooth, we have $M_T=\{\pm x_0\}$ for some $x_0\in S_{\mathbb{X}}.$ Let $H$ be a hyperspace such that $x_0\perp_B H.$ From \cite[Th. 2.6]{Sa} it follows that there exists a natural number $n$ such that $A_n\in S_{\mathbb{K}(\mathbb{X},\mathbb{Y})}$ is a non-trivial uniform $\epsilon-$BPB approximation of $T,$ where $A_n$ is defined as follows:
	 $$A_n(\alpha x_0+h)=\alpha Tx_0+(1-\frac{1}{n})Th,~\text{ where~} \alpha \in \mathbb{R}, h\in H.$$
	   We only show that $M_{A_n}=\{\pm x_0\}.$ Let $z\in M_{A_n}.$ Then $z=\alpha x_0+h,$ for some $\alpha \in \mathbb{R}$ and $h\in H.$ Now, $1=\|z\|=\|\alpha x_0+h\|\geq \|\alpha x_0\|=|\alpha|$ and 
	\begin{eqnarray*}
	1=\|A_nz\|&=&\|\alpha Tx_0+(1-\frac{1}{n})Th\|\\
	&=&\|(1-\frac{1}{n})T(\alpha x_0+h)+\frac{1}{n}\alpha Tx_0\|\\
	&\leq & (1-\frac{1}{n})\|T(\alpha x_0+h)\|+\frac{1}{n}|\alpha|\|Tx_0\|\\
	&\leq& (1-\frac{1}{n})\|Tz\|+\frac{1}{n}|\alpha|\\
	&\leq &(1-\frac{1}{n})+\frac{1}{n}\\
	&\leq & 1.
	\end{eqnarray*}
This implies that $\|Tz\|=1,$ i.e., $z\in M_T=\{\pm x_0\}.$ Thus, $M_{A_n}=\{\pm x_0\}.$ Since $\mathbb{Y}$ is smooth, $A_nx_0$ is smooth. Hence, by \cite[Th. 4.1]{PSG} $A_n$ is smooth. This completes the proof of the theorem.
\end{proof}

\begin{theorem}
	Let $\mathbb{X}, \mathbb{Y}$ be Banach spaces such that $dim(\mathbb{Y})>1$. Let $T\in S_{\mathbb{L}(\mathbb{X},\mathbb{Y})}$ be such that $M_T=\{\pm x_0\}.$ Suppose that $\sup\{\|Tx\|:x\in C\}<1$ for all closed subset $C$ of $S_{\mathbb{X}}$ with $d(\pm x_0, C)>0.$ Then for each $\epsilon>0,$ $T$ admits a non-trivial uniform $\epsilon-$BPB approximation $A_{\epsilon}$ such that $M_{A_\epsilon}=M_T.$
\end{theorem}
\begin{proof}
	Let $\epsilon>0$ be given. Let $H$ be a hyperspace of $\mathbb{X}$ such that $x_0\perp_B H.$ We first show that there exists $\delta(\epsilon)>0$ such that $M_T(\delta(\epsilon))\subseteq B(x_0,\epsilon)\cup B(-x_0,\epsilon).$ Let $C=S_{\mathbb{X}}\setminus \{B(x_0,\epsilon)\cup B(-x_0,\epsilon)\}.$ Then $C$ is a closed subset of $S_{\mathbb{X}}$ with $d(\pm x_0,C)>0.$ Therefore, $\sup\{\|Tx\|:x\in C\}<1.$ Choose $\delta>0$ such that $\sup\{\|Tx\|:x\in C\}<1-\delta.$ Hence, if $z\in M_T(\delta),$ then $\|Tz\|>1-\delta,$ i.e., $z\notin C.$ Thus, $z\in B(x_0,\epsilon)\cup B(-x_0,\epsilon).$ Let $\delta(\epsilon)=\delta.$ Then $M_T(\delta(\epsilon))\subseteq B(x_0,\epsilon)\cup B(-x_0,\epsilon).$ Now, if $T$ is rank one operator, then proceeding similarly as in Theorem \ref{th-rank1}, we can show that $T$ admits a non-trivial uniform $\epsilon-$BPB approximation $A_\epsilon$ such that $M_{A_\epsilon}=M_T.$  If rank$(T)>1,$ then proceeding similarly as in \cite[Th. 2.6]{Sa} it can be shown that there exits a natural number $n$ such that $A_n\in S_{\mathbb{L}(\mathbb{X},\mathbb{Y})}$ is a non-trivial uniform $\epsilon-$BPB approximation of $T,$ where $A_n$ is defined as follows:
	$$A_n(\alpha x_0+h)=\alpha Tx_0+(1-\frac{1}{n})Th,~\text{ where~} \alpha \in \mathbb{R}, h\in H.$$
	Proceeding similarly as in Theorem \ref{th-smooth}, it can be shown that $M_{A_n}=\{\pm x_0\}.$ This completes the proof of the theorem.
\end{proof}

As the final result of this section, we obtain an improvement of \cite[Th. 2.10]{Sa}.
\begin{theorem}
	Let $\mathbb{X}$ be a reflexive, smooth Banach space. Let $\mathcal{F}$ be the class of all norm one smooth operators in $\mathbb{K}(\mathbb{X},\mathbb{X}).$ Then the pair $(\mathbb{X},\mathbb{X})$ does not have the uniform sBPBp with respect to $\mathcal{F}.$  
\end{theorem}
\begin{proof}
	Since $\mathbb{X}$ is a reflexive, smooth Banach space, set of all extreme points of $B_{\mathbb{X}}$ is non-empty. Again the set of all exposed points is dense in the set of all extreme points \cite{K}. Suppose $x_0$ is an exposed point of $B_{\mathbb{X}}.$ Then there exists a hyperspace $H_0$ of $\mathbb{X}$ such that $x_0\perp_{SB}H_0.$ Now, for each $n \in \mathbb{N},$ define $A_n:\mathbb{X}\to\mathbb{X}$ by
	 $$A_n(\alpha x_0+h)=\alpha x_0+(1-\frac{1}{n})h,~\text {where}~ \alpha \in \mathbb{R}, h\in H_0.$$ 
	 Observe that if $\alpha \neq 0,h\neq 0,$ then $\|\alpha x_0+h\|>\|\alpha x_0\|,$ since $x_0\perp_{SB}h.$ If $\alpha x_0+h\in S_{\mathbb{X}},$ then
	 \begin{eqnarray*}
	  \|A_n(\alpha x_0+h)\|&=&\|\alpha x_0+(1-\frac{1}{n})h\|	\\
	                      &=& \|(1-\frac{1}{n})(\alpha x_0+h)+\frac{1}{n}\alpha x_0\|\\
	                      &\leq& (1-\frac{1}{n})\|\alpha x_0+h\|+\frac{1}{n}\|\alpha x_0\|\\
	                      &\leq& (1-\frac{1}{n})\|\alpha x_0+h\|+\frac{1}{n}\|\alpha x_0+h\|,~(\text{equality~holds~iff~}h=0)\\	
	                      &=&\|\alpha x_0+h\|.
	 \end{eqnarray*}
     Therefore, for each $n,$ $\|A_n\|=1$ and $M_{A_n}=\{\pm x_0\}.$ Since $\mathbb{X}$ is reflexive, smooth Banach space, by \cite[Th. 4.1]{PSG}, $A_n$ is smooth in $\mathbb{K}(\mathbb{X},\mathbb{X})$ for each $n,$ i.e., $A_n\in \mathcal{F}.$ Suppose $h_0\in H_0\cap S_{\mathbb{X}}.$ Clearly, $\|x_0+h_0\|> 1$ and $\|x_0-h_0\|> 1.$ If possible, suppose the pair $(\mathbb{X},\mathbb{X})$ has uniform sBPBp with respect to $\mathcal{F}.$ Choose $\epsilon\in (0,1).$ Let $\eta(\epsilon)$ be the constant of uniform sBPBp of $(\mathbb{X},\mathbb{X})$ with respect to $\mathcal{F}.$ Since $\|A_nh_0\|\to 1,$ there exists $k\in \mathbb{N}$ such that $\|A_kh_0\|>1-\eta(\epsilon).$ Since $M_{A_k}=\{\pm x_0\},$ either $\|x_0-h_0\|<\epsilon$ or $\|x_0+h_0\|<\epsilon,$ which contradicts the choice of $\epsilon.$ Therefore, the pair $(\mathbb{X},\mathbb{X})$ does not have the uniform sBPBp with respect to $\mathcal{F}.$ 
\end{proof}

 We end this section with an open question that we could not answer. In view of the results obtained here, it is apparent that an answer to the following question will be useful in having a better understanding of the uniform $\epsilon-$BPB approximations of operators as well as the norm attainment set of an operator.\\
 
 \textbf{Question:}
 Is it true that each finite-dimensional Banach space $\mathbb{X}$ satisfies Property $(P)?$ What about the infinite-dimensional case? \\

\section{Applications of uniform $\epsilon-$BPB approximation on some particular Banach spaces}

In this section we deal with operators on some well studied Banach spaces like $\ell_{\infty}^{n}$ and $\ell_{1}^{n},$ in view of the preservation of the norm attainment set under uniform $\epsilon-$BPB approximations. We note that in case of norm one operators which are not extreme contractions, the situation is completely understood by virtue of Theorem \ref{th-notextreme}. Therefore, we restrict ourselves mostly to the case of extreme contractions. We begin with extreme contractions in $\mathbb{L}(\ell_{\infty}^n, \ell_{\infty}^n ).$ First we require the following lemma which can be proved by using \cite[Lemma 2.8]{Sh} and the fact that on a finite-dimensional Banach space, an operator is an extreme contraction if and only if the adjoint of the operator is an extreme contraction.
\begin{lemma}\label{lemma-infinity}
	Let $\mathbb{X}=\ell_{\infty}^n$ and $T\in S_{\mathbb{L}(\mathbb{X},\mathbb{X})}.$  If $T$ is an extreme contraction, then each row of corresponding matrix with respect to standard ordered basis contains exactly one non-zero element from $S=\{1, -1\}$ and remaining elements in the same row are all zero.
\end{lemma}

We are now ready to present the desired result in $\mathbb{L}(\ell_{\infty}^n, \ell_{\infty}^n).$

\begin{theorem}\label{infinity-contraction}
	Suppose $\mathbb{X}=\ell_{\infty}^n$ and $T\in S_{\mathbb{L}(\mathbb{X},\mathbb{X})}$ is an extreme contraction which is not an isometry. Then $T$ admits a non-trivial uniform $\epsilon-$BPB approximation $A_{\epsilon}$ with $M_T=M_{A_{\epsilon}},$  for each $\epsilon>0.$
\end{theorem}
\begin{proof}
	Let $\epsilon>0$ be fixed.
	Let $ T =\big (t_{ij}\big)_{n\times n} $ be the matrix representation of $ T $ with respect to the standard ordered basis of $ \ell_{\infty}^{n}.$
	Since $T$ is an extreme contraction, using Lemma \ref{lemma-infinity}, we say that each row of $T$ contains exactly one non-zero element from $S=\{1,-1\}$ and remaining elements in the same row are zero. Now, since $T$ is not an isometry, it is easy to see that one of the columns of $T$ contains at least two non-zero elements from $S=\{1, -1\}.$ Without loss of generality, we may and do assume that $t_{11}$ and $t_{21}$ are non-zero. 
	We define 
	$ A_\epsilon =\big (a_{ij}\big)_{n\times n} ,$ where  
	\[ a_{ij} =  \begin{cases}
	t_{ij},& (i,j) \neq (1,1)\\
	t_{11}-\frac{\epsilon}{2},&(i,j) = (1,1) ~\& ~t_{11}=1\\
	t_{11}+\frac{\epsilon}{2},&(i,j) = (1,1) ~\& ~t_{11}=-1.
	\end{cases} \] 
	Clearly, $A_\epsilon\neq T.$
	By a simple calculation, it is easy to see that $\|A_{\epsilon}\|=1,$  $M_T=M_{A_{\epsilon}}$ and $\|T-A_\epsilon\|=\frac{\epsilon}{2}<\epsilon.$ 	Since $\mathbb{X}$ is finite-dimensional, it is reflexive, Kadets-Klee. Thus, by Theorem \ref{th-appnorm}, there exists $\delta(\epsilon)>0$ such that $ M_T(\delta(\epsilon))\subset \bigcup_{z\in M_T} B(z,\epsilon).$ Now, choose $x\in S_{\mathbb{X}}$ such that $\|Tx\|>1-\delta(\epsilon).$ Then there exists $z\in M_T$ such that $x\in B(z,\epsilon),$ i.e., $z\in M_{A_\epsilon}$ and $\|z-x\|<\epsilon.$ Thus, $A_\epsilon$ is a non-trivial uniform $\epsilon-$BPB approximation of $T.$ This completes the proof of the theorem.
\end{proof}

As an application of the above theorem, we characterize the isometries in $\mathbb{L}(\ell_{\infty}^n, \ell_{\infty}^n)$ in terms of uniform $\epsilon-$BPB approximations.

\begin{cor}\label{cor-infinity}
	Let $\mathbb{X}=\ell_{\infty}^n$ and $T\in S_{\mathbb{L}(\mathbb{X},\mathbb{X})}.$  Then $T$ is an isometry if and only if there exists $\epsilon_0>0$ such that for each $0<\epsilon<\epsilon_0,$ $T$ is the only uniform $\epsilon-$BPB approximation of $T.$ Moreover, the pair $(\ell_{\infty}^n,\ell_{\infty}^n),~(n\in \mathbb{N})$ has non-trivial norm preserving uniform BPB approximation property. 
\end{cor}
\begin{proof}
	The necessary part of the corollary follows from Corollary \ref{cor-isometry}. For the sufficient part, suppose that $T$ is not an isometry. If $T$ is an extreme contraction, then the result follows from Theorem \ref{infinity-contraction}. If $T$ is not an extreme contraction, then the result follows from Theorem \ref{th-notextreme}.\\
	 The rest of the proof now follows directly from Definition \ref{def-pair}.
\end{proof}

We now obtain some results on uniform $\epsilon-$BPB approximations in $\mathbb{L}(\ell_1^n, \ell_1^n),$ in the same spirit as above. We refer to  \cite[Lemma 2.8]{Sh} for the proof of the following lemma, which we require for our purpose.

\begin{lemma}\label{one}
	Let $\mathbb{X}=\ell_{1}^n$ and $T\in S_{\mathbb{L}(\mathbb{X},\mathbb{X})}.$ If $T$ is an extreme contraction, then each column of corresponding matrix with respect to standard ordered basis contains exactly one non-zero element from $S=\{1, -1\}$ and remaining elements in the same column are all zero.
\end{lemma}

Our next two results are analogous to Theorem \ref{infinity-contraction} and Corollary \ref{cor-infinity}, respectively.

\begin{theorem}\label{th-lone}
	Suppose $\mathbb{X}=\ell_{1}^n$ and $T\in S_{\mathbb{L}(\mathbb{X},\mathbb{X})}$ is an extreme contraction which is not an isometry. Then $T$ admits a non-trivial uniform $\epsilon-$BPB approximation $A_{\epsilon}$ with $M_T=M_{A_{\epsilon}},$  for each $\epsilon >0.$  
\end{theorem}
\begin{proof}
	Let $\epsilon>0$ be fixed.
	Let $ T =\big (t_{ij}\big)_{n\times n} $ be the matrix representation of $ T $ with respect to the standard ordered basis of $ \ell_{1}^{n}.$ Since $T$ is an extreme contraction, using Lemma \ref{one}, we say that each column of $T$ contains exactly one non-zero element from $S=\{1,-1\}$ and the remaining elements in the same column are all zero. Now, since $T$ is not an isometry, it is easy to see that one of the rows of $T$ contains at least two non-zero elements from $S=\{1, -1\}.$ Without loss of generality, we may and do assume that $t_{11}$  and $t_{12}$ are non-zero. Using Lemma \ref{one} again, it is easy to observe that at least one row of $T$ is zero. Without loss of generality, we assume that  $t_{2j}=0$ for all $1\leq j\leq n.$\\
	We define 
	$ A_\epsilon =\big (a_{ij}\big)_{n\times n} ,$ where  
	\[ a_{ij} =  \begin{cases}
	t_{ij},& (i,j) \neq (1,1) , (2,1)\\
	t_{11}-\frac{\epsilon}{4},& (i,j) = (1,1) ~\& ~t_{11}=1\\
	t_{11}+\frac{\epsilon}{4},& (i,j) = (1,1) ~\& ~t_{11}=-1\\
	\frac{\epsilon}{4}, & (i,j) = (2,1).
	\end{cases} \] 
	Clearly, $A_\epsilon\neq T$ and $\|A_\epsilon-T\|<\epsilon.$ Now, by simple calculation, it is easy to see that $\|A_{\epsilon}\|=1$ and $M_T=M_{A_{\epsilon}}.$  Rest of the proof follows using similar arguments as in Theorem \ref{infinity-contraction}.
\end{proof}

\begin{cor}\label{cor-l1}
	Let $\mathbb{X}=\ell_{1}^n$ and $T\in S_{\mathbb{L}(\mathbb{X},\mathbb{X})}.$ Then $T$ is an isometry if and only if there exists $\epsilon_0>0$ such that for each $0<\epsilon<\epsilon_0,$ $T$ is the only uniform $\epsilon-$BPB approximation of $T.$ Moreover, the pair $(\ell_{1}^n,\ell_{1}^n),~(n\in \mathbb{N})$ has non-trivial norm preserving uniform BPB approximation property.
\end{cor}
\begin{proof}
	The necessary part of the corollary follows from Corollary \ref{cor-isometry}. For the sufficient part, suppose that $T$ is not an isometry. If $T$ is an extreme contraction, then the result follows from Theorem \ref{th-lone}. If $T$ is not an extreme contraction, then the result follows from Theorem \ref{th-notextreme}.\\
	 The rest of the proof now follows directly from Definition \ref{def-pair}.
\end{proof}

For $n \geq 4,$ extreme contractions in $\mathbb{L}(\ell_{\infty}^n, \ell_1^n)$ are yet to be completely identified. We note that the case $n=2$ has already been covered in Theorem \ref{infinity-contraction}, since $\ell_{\infty}^2$ is isometrically isomorphic to $\ell_1^2.$ We would like to study the preservation of operator norm attainment for extreme contractions in $\mathbb{L}(\ell_{\infty}^3, \ell_1^3)$ under uniform $\epsilon-$BPB approximations. To this end, we begin with a useful remark from \cite{BM}.

\begin{remark}
	We say that two $n\times n$ matrices $A$ and $B$ are equivalent if there exist diagonal matrices $D_1,D_2$ with unimodular diagonal entries and permutation matrices $P_1,P_2$ such that $B=D_1P_1AP_2D_2.$ 
\end{remark}

Now, the promised theorem in $\mathbb{L}(\ell_{\infty}^3, \ell_1^3)$:

\begin{theorem}\label{th-infinity1}
	Suppose $\mathbb{X}=\ell_{\infty}^3$ and $\mathbb{Y}=\ell_{1}^3.$ Let $T\in S_{\mathbb{L}(\mathbb{X},\mathbb{Y})}$ be an extreme contraction. Then $T$ admits a non-trivial uniform $\epsilon-$BPB approximation $A_{\epsilon}$ with $M_T=M_{A_{\epsilon}},$  for each $\epsilon>0.$
\end{theorem}
\begin{proof}
	Let $\epsilon>0$ be fixed. From \cite{L} it follows that $\mathbb{L}(\mathbb{X},\mathbb{Y})$ has $90$ extreme contractions. A detailed calculation reveals that (see also \cite[Th. 3]{BM}) out of these $90$ extreme contractions, $18$ extreme contractions are equivalent to $\begin{pmatrix}
	1 &  0 \\ 
	0 &  0 
	\end{pmatrix}\oplus 0$ and $72$ extreme contractions are equivalent to $\begin{pmatrix}
	\frac{1}{2} &  \frac{1}{2} \\ 
	\frac{1}{2} &  -\frac{1}{2}
	\end{pmatrix}\oplus 0.$ Therefore, upto equivalence, 
	either of the following two cases holds :\\
	$\textbf{Case 1.}$  $T$ is of the form $\begin{pmatrix}
	1 &  0 & 0\\ 
	0 &  0 & 0\\
	0 & 0 & 0
	\end{pmatrix},$\\
	$\textbf{Case 2.}$  $T$ is of the form 	$\begin{pmatrix}
	\frac{1}{2} &  \frac{1}{2} & 0\\ 
	\frac{1}{2} &  -\frac{1}{2} & 0\\
	0 & 0 & 0
	\end{pmatrix}.$
	
	We note that in $\textbf{Case 1},$ $T$ is a rank one operator. Therefore, using Theorem \ref{th-rank1}, we get non-trivial uniform $\epsilon-$BPB approximation $A_\epsilon$ such that $M_T=M_{A_\epsilon}.$\\ For $\textbf{Case 2},$ there are $72$ such extreme contractions. We prove the theorem for one such operator, namely, 	$T=\begin{pmatrix}
	\frac{1}{2} &  \frac{1}{2} & 0\\ 
	\frac{1}{2} &  -\frac{1}{2} & 0\\
	0 & 0 & 0
	\end{pmatrix}.$ We can prove the same for remaining extreme contractions by same technique.
	We define 
		$A_\epsilon=\begin{pmatrix}
	\frac{1}{2}-\frac{\epsilon}{8} & ~ \frac{1}{2}-\frac{\epsilon}{8}~ & 0\\ 
	\frac{1}{2} & ~ -\frac{1}{2}~ & 0\\
	\frac{\epsilon}{8}~ & \frac{\epsilon}{8}~ & 0
	\end{pmatrix}.$
	Clearly, $T\neq A_{\epsilon},$ and it is easy to verify  that $\|A_{\epsilon}\|=1,$  $M_T=M_{A_{\epsilon}}$ and $\|T-A_\epsilon\|=\frac{\epsilon}{2}<\epsilon.$ 
	Rest of the proof follows using similar arguments as in Theorem \ref{infinity-contraction}. 
\end{proof}

The proof of the following corollary is now obvious in light of Theorem \ref{th-notextreme} and Theorem \ref{th-infinity1}, and is therefore omitted.

\begin{cor}\label{cor-linfinityl1}
	The pair $(\ell_{\infty}^3,\ell_{1}^3),$ has non-trivial norm preserving uniform BPB approximation property. 
\end{cor}

We are now interested in studying the case of $\ell_p^2$ spaces in the same spirit. The next result supplies us with a large class of two-dimensional Banach spaces with Property (P).

\begin{theorem}\label{th-propertyplp}
	The Banach space $\ell_p^2~(p>2,p\in \mathbb{N})$ satisfies Property $(P).$ 
\end{theorem}
\begin{proof}
 From  \cite[Th. 2.8]{S}, we get for any $A\in S_{\mathbb{L}(\ell_p^2,\ell_p^2)}\setminus \text{Iso}(\ell_p^2,\ell_p^2),$ $|M_A|\leq 2(8p-5).$ Since $S_{\ell_p^2}$ has a finite length, we can choose $r_0>0$ such that for arbitrary $x_i\in S_{\ell_p^2}, ~1\leq i\leq 2(8p-5),$ $S_{\ell_p^2}\setminus \cup_{1\leq i\leq 2(8p-5)}B(x_i,r_0)\neq \emptyset.$ Suppose that $M_A=\{\pm x_1,\pm x_2,\ldots,\pm x_k\}.$ Then $2k\leq 2(8p-5).$ Clearly, $S_{\ell_p^2}\setminus \cup_{1\leq i\leq k}B(\pm x_i,r_0)\neq \emptyset.$ Choose $z\in S_{\ell_p^2}\setminus \cup_{1\leq i\leq k}B(\pm x_i,r_0).$ Then we have $B(z,r_0)\cap M_A=\emptyset.$ Thus, the space $\ell_p^2$ satisfies Property $(P).$
\end{proof}

In the next proposition, we characterize the norm attainment set and the minimum norm attainment set of a particular operator that turns out to be useful in our study. 

 \begin{prop}\label{Clarkson} 
	Let $ \mathbb{X} = \ell_{p}^{2}(p>1)  $. Let $ T \in \mathbb{L}(\mathbb{X},\mathbb{X}) $ be such that $ T=
	\begin{pmatrix}
	1 & 1 \\ 
	\rule{0em}{3ex}1 &  -1
	\end{pmatrix}.$ Then the followings hold true:\\
	$(i)$  $M_T=\{\pm(1,0), \pm(1,0)\}$ if $1<p<2.$\\
	$(ii)$ $M_T=S_{\mathbb{X}}$ if $p=2.$\\
	$(iii)$ $ M_T=\{\pm(\frac{1}{2^{1/p}}, \frac{1}{2^{1/p}}),\pm(\frac{1}{2^{1/p}}, -\frac{1}{2^{1/p}})\}$ if $p>2.$  
\end{prop}
\begin{proof}
	For $p=2,$ $ T$ is a scalar multiple of an isometry. Therefore, $(ii)$ holds true.\\
	Rest of the proof follows from the Clarkson inequality \cite{C}.
\end{proof}

In \cite[Th. 2.9]{Sa}, it was proved that given any isometry $T\in \mathbb{L}(\ell_p^2,\ell_p^2),$ where $p\in \mathbb{N}\setminus\{1,2\},$ $T$ is the only uniform $\epsilon-$BPB approximation of $T,$ for sufficiently small $\epsilon>0.$ Although the result stands, the proof of it, as given in \cite{Sa}, contains a subtle logical gap. In the next theorem, we would like to cover this gap and also to obtain a quantitative version of the said result. To this end, let us first consider the following definition.   

\begin{definition}
	Let $\|.\|_p$ denotes the usual $\ell_p$ norm on $\mathbb{R}^2.$ Let $C$ be a closed rectifiable curve in $\mathbb{R}^2.$ Given any two points $x,y\in C,$ let $s(x,y)$ denote the Euclidean arc length of the shorter arc between $x$ and $y,$ along the curve $C.$ Let $\epsilon>0$ be given. A real number $\delta(=\delta(\epsilon))>0$ is said to be an $\epsilon-$arc length constant of $C$ corresponding to $\|.\|_p,$ if given any two points $x,y\in C,~s(x,y)\geq \epsilon$ implies that $\|x-y\|_p\geq \delta.$
\end{definition}

We are now ready to present the promised quantitative version of \cite[Th. 2.9]{Sa}.

\begin{theorem}
	Let $\mathbb{X}=\ell_p^2,$ where $p\in \mathbb{N}\setminus \{1,2\}.$ Let $T\in \mathbb{L}(\mathbb{X},\mathbb{X})$ be an isometry. Then there exists $\epsilon_0>0$ such that for each $0<\epsilon<\epsilon_0,~ T$ is the only uniform $\epsilon-$BPB approximation of $T.$ Moreover, we can choose $\epsilon_0=\min\{2^{\frac{p-1}{p}},\delta_1\},$ where $\delta_1=\delta(\frac{L}{2(16p-9)})$ is an $\frac{L}{2(16p-9)}-$arc length constant of the curve $C=\{(x,y)\in \mathbb{R}^2:|x|^p+|y|^p=1\},$ and $L$ is the length of the curve $C$ with respect to $\|.\|_2.$  
\end{theorem}

\begin{proof}
	From \cite[Th. 4.2]{CL} (see also \cite[Th. 2.11]{CSS}), we get\\
	
	$\text{Iso}(\mathbb{X},\mathbb{X})=\bigg\{ \pm\begin{pmatrix}
	1 & 0 \\ 
	0 &  1
	\end{pmatrix} ,\pm \begin{pmatrix}
	-1 & 0 \\ 
	0 &  1
	\end{pmatrix} ,\pm\begin{pmatrix}
	0 & 1 \\ 
	1 &  0
	\end{pmatrix} ,\pm \begin{pmatrix}
	0 & -1 \\ 
	1 &  0
	\end{pmatrix}  \bigg\}.$\\
	Let $T_1, T_2\in \text{Iso}(\mathbb{X},\mathbb{X})$ such that $T_1\neq T_2.$ Then by a simple calculation and using Proposition \ref{Clarkson}, it is easy to see that $ \|T_1-T_2\|= 2^{\frac{p-1}{p}}~~\text{or}~~ 2.$ Thus, $ min\big\{\|T_1-T_2\|:T_1, T_2\in \text{Iso}(\mathbb{X},\mathbb{X}),T_1\neq T_2\big\}= 2^{\frac{p-1}{p}}.$
 Let $\epsilon_0=\min\{2^{\frac{p-1}{p}},\delta_1\},$ where $\delta_1$ is as mentioned in the theorem. Let $0<\epsilon<\epsilon_0.$ If possible, let $A_\epsilon$ be a non-trivial uniform $\epsilon-$BPB approximation of $T$ and the corresponding constant is $\delta(\epsilon).$ Then $0<\|T-A_\epsilon\|<\epsilon<2^{\frac{p-1}{p}}$ implies that $A_\epsilon$ is not an isometry. Hence, by \cite[Th. 2.8]{S}, $|M_{A_\epsilon}|\leq 2(8p-5).$ Let $P=\{x_1,x_2,\ldots,x_{16p-9}\}$ be a partition of $C$ such that $s(x_{16p-9},x_1)=s(x_i,x_{i+1})=\frac{L}{16p-9}$ for all $1\leq i\leq 16p-8.$ Since $|M_{A_\epsilon}|\leq 16p-10,$ without loss of generality, we may assume that the shorter arc between $x_1$ and $x_2$ along the curve $C$ does not contain any vector of $M_{A_\epsilon}.$ Choose a vector $z\in C$ such that $s(x_1,z)=s(z,x_2)=\frac{L}{2(16p-9)}.$ Clearly, $z$ is in the shorter arc between $x_1$ and $x_2.$ Now, by definition of arc length constant, we have, $\|x_1-z\|\geq \delta(\frac{L}{2(16p-9)})=\delta_1$ and  $\|x_2-z\|\geq \delta(\frac{L}{2(16p-9)})=\delta_1.$ Thus, $B(z,\delta_1)\cap M_{A_\epsilon}=\emptyset.$ Since $T$ is an isometry, $z\in M_T,$ i.e., $1=\|Tz\|>1-\delta(\epsilon).$ Therefore, there exists $x\in M_{A_\epsilon}$ such that $\|z-x\|<\epsilon<\delta_1.$ This contradicts that $B(z,\delta_1)\cap M_{A_\epsilon}=\emptyset.$ Thus, $T$ is the only uniform $\epsilon-$BPB approximation of $T.$ This completes the proof of the theorem. 
\end{proof}

\section{Uniform $\epsilon-$BPB approximation on $\mathbb{L}(\mathbb{H}_1,\mathbb{H}_2)$}

After discussing the uniform $\epsilon-$BPB approximations of operators between general (and also some particular) Banach spaces, we would like to study the same under the more restrictive setting of Hilbert spaces. As mentioned in the introduction, the norm attainment set of an operator assumes a very special form in this case and we will see that it allows us to draw several stronger conclusions in comparison to the case of operators between Banach spaces. We begin with the observation that every Hilbert space satisfies Property (P).

\begin{prop}\label{th-propertyphil}
	Let $\mathbb{H}$ be a Hilbert space. Then $\mathbb{H}$ satisfies Property $(P).$
\end{prop}
\begin{proof}
	Let $A\in S_{\mathbb{L}(\mathbb{H},\mathbb{H})}\setminus \text{Iso}(\mathbb{H},\mathbb{H}).$ If $A$ does not attain its norm, then for any $x\in S_\mathbb{H},B(x,1)\cap M_A=\emptyset.$ Let $M_A\neq \emptyset.$ Then by \cite[Th. 2.2]{SP} we can say that there exists a subspace $H_0$ of $\mathbb{H}$ such that $M_A=S_{H_0}.$ Since $A$ is not an isometry, $H_0\neq \mathbb{H}.$ Therefore, $H_0^\perp\neq \{\theta\}.$ Choose $x\in S_{H_0^\perp}.$ Then for each $y\in M_A, x\perp y,$ i.e., $\|x-y\|\geq 1.$ Therefore, $B(x,1)\cap M_A=\emptyset.$ This completes the proof of the proposition. 
\end{proof}

We discuss the exact preservation of operator norm attainment set under uniform $\epsilon-$BPB approximations in the following theorem.

\begin{theorem}\label{th-hil}
	Let $\mathbb{H}_1,\mathbb{H}_2$ be Hilbert spaces such that $dim(\mathbb{H}_2)>1$. Suppose $T\in S_{\mathbb{L}(\mathbb{H}_1,\mathbb{H}_2)}$ is such that $M_T=S_{H_0},$ where $H_0$ is a subspace of $\mathbb{H}_1.$ Then for every $\epsilon>0,$ $T$ admits a non-trivial uniform $\epsilon-$BPB approximation $A_\epsilon$ such that $M_T=M_{A_\epsilon}$ if and only if $\|T\|_{H_0^\perp}<1.$  
\end{theorem}
\begin{proof}
	The necessary part of the theorem follows from Proposition \ref{prop-dist}. We only prove the sufficient part of the theorem. Let $\|T\|_{H_0^\perp}<1.$ Let $\epsilon>0$ be given. We first show that there exits $\delta(\epsilon)>0$ such that $M_T(\delta(\epsilon))\subseteq \cup_{x\in M_T}B(x,\epsilon).$ If it is not true, then for each $n\in \mathbb{N},$ there exists $z_n\in M_T(\frac{1}{n})$ such that $z_n\notin \cup_{x\in M_T}B(x,\epsilon).$ Clearly, $\|Tz_n\|>1-\frac{1}{n},$ i.e., $\|Tz_n\|\to 1.$  Now, for each $n\in \mathbb{N},$ $z_n=x_n+y_n,$ where $x_n\in H_0$ and $y_n\in H_0^\perp.$ $1=\|z_n\|^2=\|x_n\|^2+\|y_n\|^2$ implies that $\{\|x_n\|\}$ and $\{\|y_n\|\}$ are bounded sequences of real numbers. Therefore, there exist convergent subsequences of $\{\|x_n\|\}$ and $\{\|y_n\|\}.$ Without loss generality, we may assume that $\{\|x_n\|\}$ and $\{\|y_n\|\}$ are convergent. Since $x_n\perp y_n$ and $x_n\in H_0,$ from \cite[Th. 2.3]{S}, we have, $Tx_n\perp Ty_n.$ Now,
	\begin{eqnarray*}
		&&\|Tz_n\|^2=\|Tx_n\|^2+\|Ty_n\|^2\\
		&\Rightarrow &\|Tz_n\|^2=\|T\|^2\|x_n\|^2+\|Ty_n\|^2\\
		&\Rightarrow& 1=\lim\|x_n\|^2+\lim\|Ty_n\|^2, ~\text{taking ~limit}~n\to\infty\\
		&\Rightarrow& 1-\lim \|x_n\|^2= \lim \|Ty_n\|^2\\
		&\Rightarrow& \lim \|y_n\|^2= \lim \|Ty_n\|^2.
	\end{eqnarray*} 
	Since  $\|T\|_{H_0^\perp}<1,$ there does not exist any non-zero subsequence of $\{\|y_n\|\}.$ Without loss of generality, we may assume that $y_n=0$ for all $n.$ Thus, for all $n,$ $z_n=x_n.$ Now, $\|z_n\|=1$ implies that $z_n\in M_T\subseteq \cup_{x\in M_T}B(x,\epsilon),$ a contradiction. Therefore, there exists $\delta(\epsilon)>0$ such that $M_T(\delta(\epsilon))\subseteq \cup_{x\in M_T}B(x,\epsilon).$
	Now, we consider two cases.\\
	
	\textbf{Case 1: $\|T\|_{H_0^\perp}>0.$}\\
	For each $n\in\mathbb{N}\setminus\{1\},$ define $A_n\in\mathbb{L}(\mathbb{H}_1,\mathbb{H}_2)$ by $$A_n(u+v)=Tu+(1-\frac{1}{n})Tv,~ \text{where}~ u\in H_0,v\in H_0^\perp.$$
	Since $\|T\|_{H_0^\perp}>0,$ there exists $v\in H_0^\perp$ such that $Tv\neq 0.$ Hence, $A_n\neq T$ for all $n.$ It is easy to observe that $\|A_n\|=1,$ $M_{A_n}=M_T$ and $A_n\to T.$ Therefore, there exists $k\in\mathbb{N}\setminus\{1\}$ such that $\|A_k-T\|<\epsilon.$ Now, let $z\in S_{\mathbb{H}_1}$ be such that $\|Tz\|>1-\delta(\epsilon).$ Then $z\in M_T(\delta(\epsilon))\subseteq \cup_{x\in M_T}B(x,\epsilon).$ Therefore, there exists $x\in M_T=M_{A_k}$ such that $z\in B(x,\epsilon),$ i.e., $\|z-x\|<\epsilon.$ Thus, $A_k$ is a non-trivial uniform $\epsilon-$BPB approximation of $T$ such that $M_{A_k}=M_T.$ \\
	
	\textbf{Case 2: $\|T\|_{H_0^\perp}=0.$}\\
	If $H_0$ is one-dimensional, then rank$(T)=1$ and proceeding similarly as in Theorem \ref{th-rank1} we can show that $T$ admits a non-trivial uniform $\epsilon-$BPB approximation $A_\epsilon$ such that $M_{A_\epsilon}=M_T.$ Assume that $H_0$ is not one-dimensional. Suppose that $\mathcal{B}=\{e_\alpha:\alpha\in \Lambda\}$ is a complete orthonormal basis of $\mathbb{H}_1.$ For simplicity assume that $1,2\in \Lambda$ and $e_1,e_2\in H_0.$ Choose $u_1=aTe_1+bTe_2\in S_{\mathbb{H}_2}\cap B(Te_1,\frac{\epsilon}{2}),$ where $a,b\in \mathbb{R}.$ Let $u_2=-bTe_1+aTe_2.$ Then $u_1\perp u_2$ and $u_2\in S_{\mathbb{H}_2}\cap B(Te_2,\frac{\epsilon}{2}).$
	
	Now, define $A_\epsilon:\mathbb{H}_1\to\mathbb{H}_2$ in the following way:
	\begin{eqnarray*}
		A_\epsilon e_1 &=& u_1,\\
		A_\epsilon e_2 &=& u_2,\\
		A_\epsilon e_\alpha &=& Te_\alpha~ \text{for~all~} \alpha \in \Lambda \setminus \{1,2\}.
	\end{eqnarray*}
	It is easy to observe that $A_\epsilon$ can be extended to a bounded linear operator from $\mathbb{H}_1$ to $\mathbb{H}_2.$ Clearly, $M_{A_\epsilon}=S_{H_0}=M_T$ and $\|T-A_\epsilon\|<\epsilon.$ Now, proceeding similarly as in \textbf{Case 1} it can be shown that $A_\epsilon$ is a non-trivial uniform $\epsilon-$BPB approximation of $T.$ 
	This completes the proof of the sufficient part. 
\end{proof}

We would like to point out that given   $T\in S_{\mathbb{L}(\mathbb{H})}$  and $\epsilon>0,$ we can find $A_{\epsilon}$  which is a non-trivial uniform $\epsilon-$BPB approximation of $T$ with  $M_T\neq M_{A_{\epsilon}}.$ Here is one such example.

\begin{example}
	Let $\mathbb{X}=\ell_{2}^2$  and  $T=\begin{pmatrix}
	1 & 0 \\ 
	0 &  0
	\end{pmatrix}.$ Then $M_T=\{\pm (1,0)\}.$ Let $\epsilon>0$ be given. Choose $\theta\in (0,\frac{\pi}{2})$ such that $1-\frac{\epsilon^2}{8}<\sin \theta.$ Let
	  $A_{\epsilon}=\begin{pmatrix}
	\sin^2\theta & \sin\theta\cos\theta \\ 
	\sin\theta\cos\theta &  \cos^2\theta
	\end{pmatrix}.$
	Clearly, $T\neq A_{\epsilon}, \|A_{\epsilon}\|=1 ~~ $ and $M_{A_{\epsilon}}=\{\pm(\sin\theta,\cos\theta)\}.$ Therefore,  $M_T\neq M_{A_{\epsilon}}.$
	Now, $T-A_{\epsilon}=\begin{pmatrix}
	\cos^2\theta & -\sin\theta\cos\theta \\ 
	-\sin\theta\cos\theta &  -\cos^2\theta
	\end{pmatrix}.$ It is easy to see that $\|T-A_\epsilon\|=\cos\theta<\frac{\epsilon}{2}$. Now, by Theorem \ref{th-appnorm}, there exists $\delta>0$ such that $M_T(\delta)\subseteq B((1,0),\frac{\epsilon}{2})\cup B(-(1,0),\frac{\epsilon}{2}).$ Choose $\delta(\epsilon)=\delta.$ Let $(x,y)\in S_\mathbb{X}$ be such that $\|T(x,y)\|>1-\delta(\epsilon).$ Without loss of generality, assume that $(x,y)\in B((1,0),\frac{\epsilon}{2}).$ Now, observe that $\|(1,0)-(\sin\theta,\cos\theta)\|<\frac{\epsilon}{2}.$ Therefore, $\|(x,y)-(\sin\theta,\cos\theta)\|\leq\|(x,y)-(1,0)\|+\|(1,0)-(\sin\theta,\cos\theta)\|<\epsilon.$ Thus, $A_\epsilon$ is a non-trivial uniform $\epsilon-$BPB approximation of $T$ such that $M_T\neq M_{A_\epsilon}.$
\end{example}

Our next result shows that given an Euclidean space $ \mathbb{H} $ of any (finite) dimension, the pair $ \mathbb{L}(\mathbb{H}, \mathbb{H}) $ has non-trivial norm preserving uniform $\epsilon-$BPB approximation property.

\begin{theorem}
	The pair $(\ell_2^n,\ell_2^n) (n\in \mathbb{N})$ has non-trivial norm preserving uniform $\epsilon-$BPB approximation property.
\end{theorem}
\begin{proof}
	Let $T\in S_{\mathbb{L}(\ell_2^n,\ell_2^n)}\setminus \text{Iso}(\ell_2^n,\ell_2^n).$ By \cite[Th. 2.2]{SP}, there exists a subspace $H_0$ of $\ell_2^n$ such that $M_T=S_{H_0}.$ It is easy to observe that $\|T\|_{H_0^\perp}<1.$ Therefore, using Theorem \ref{th-hil}, we can say that for any $\epsilon>0,$ $T$ admits a non-trivial uniform $\epsilon-$BPB approximation $A_\epsilon$ such that $M_T=M_{A_\epsilon}.$ Hence, The pair $(\ell_2^n,\ell_2^n) (n\in \mathbb{N})$ has non-trivial norm preserving uniform $\epsilon-$BPB approximation property.
\end{proof}

We now obtain a complete characterization of Euclidean spaces among $ l_p^n $ spaces, in terms of uniform $\epsilon-$BPB approximations.

\begin{theorem}
	Let $\mathbb{X} = \ell_p^n, 1 < p <\infty, n \in \mathbb{N}.$ Then $p=2,$ i.e., $\mathbb{X}$  is a Hilbert space if and only if for each $\epsilon > 0$ and for each isometry $T$ on $\mathbb{X},$  $T$ has a non-trivial uniform $\epsilon$-BPB approximation which is also an isometry.
\end{theorem}
\begin{proof} First let $p=2,$ i.e., $\mathbb{X}$ is a Hilbert space. Let $T$ be an isometry on $\mathbb{X},$ then $ M_T = S _{\mathbb{X}}.$  So from Theorem \ref{th-hil}, it follows that $T$ has a non-trivial uniform $\epsilon$-BPB approximation $A$ such that $ M_T = M_A.$ So $A$ is the non-trivial uniform $\epsilon$-BPB approximation which is also an isometry. Conversely, suppose that $p \neq 2.$  Let $T$ be an isometry on $\mathbb{X}.$ From \cite[Th. 4.2]{CL}, it follows that there are only finitely many isometries on $\mathbb{X}.$  Let $\epsilon_0 = \min \{ \|T-U\|: U \in \text{Iso}(\mathbb{X}), T \neq U\}.$ Then $ \epsilon_0 > 0$. Choose $ 0 < \epsilon < \epsilon_0.$  Then it is clear that no non-trivial isometry can be a uniform $ \epsilon$-BPB approximation of $T$. This completes the proof.
\end{proof}

 In the following theorem, we deduce some necessary conditions for uniform $\epsilon-$BPB approximations for operators between Hilbert spaces, under some nice conditions.

\begin{theorem}\label{th-chhil}
	Let $\mathbb{H}_1,\mathbb{H}_2$ be Hilbert spaces. Let $T\in S_{\mathbb{L}(\mathbb{H}_1,\mathbb{H}_2)}.$ Let $M_T\neq \emptyset$ and $M_T=S_{H_0},$ where $H_0$ is a subspace of $H.$ Suppose $\sup\{\|Tz\|:z\in C\}<1$ holds for each closed subset $C$ of $S_{\mathbb{H}_1}$ with $d(M_T,C)>0.$ Then there exists $\epsilon_0>0$ such that for any $0<\epsilon<\epsilon_0,$ if $A_\epsilon$ is a uniform $\epsilon-$BPB approximation of $T,$ then the following conditions hold true:\\
	(i) $H_0^\perp \cap H=\{\theta\}$ and $H^\perp\cap H_0=\{\theta\},$ where $M_{A_\epsilon}=S_{H}$ for some subspace $H$ of $\mathbb{H}_1,$\\
	(ii) for any $\epsilon_1,\epsilon_2>0$ with $\epsilon_1^2+\epsilon_2^2>2\epsilon^2$ either $\|(T-A_\epsilon)|_{H_0}\|<\epsilon_1$ or $\|(T-A_\epsilon)|_{H_0^\perp}\|<\epsilon_2,$\\
	(iii) there exists $\delta=\delta(\epsilon)>0$ such that $M_T(\delta)\subseteq \cup_{x\in M_{A_\epsilon}}B(x,\epsilon).$ 
\end{theorem}
\begin{proof}
	Using Theorem \ref{th-ball}, we can say that there exists $\epsilon_0>0$ such that for $0<\epsilon<\epsilon_0$ if $A_{\epsilon}$ is a uniform $\epsilon-$BPB approximation of $T,$ then 
	\begin{eqnarray}\label{eqhil}
	M_{A_\epsilon}\subseteq \cup_{x\in M_T}B(x,1).
	\end{eqnarray}
	Without loss of generality, we may assume that $\epsilon_0<1.$ Let $0<\epsilon<\epsilon_0.$ Suppose  $A_{\epsilon}$ is a uniform $\epsilon-$BPB approximation of $T$ and $\delta=\delta(\epsilon)$ be the corresponding constant. Let $M_{A_\epsilon}=S_{H}$ for some subspace $H$ of $\mathbb{H}_1.$ \\
	
	(i) If possible, suppose that $H_0^\perp \cap H\neq\{\theta\}.$ Then there exists $z\in S_{\mathbb{H}_1}$ such that $z\in H_0^\perp\cap H.$ Now, $z\in S_H$ implies that $z\in M_{A_\epsilon}.$ Therefore, by (\ref{eqhil}), there exists $x\in M_T$ such that $z\in B(x,1).$ Since $z\in H_0^\perp$ and $x\in H_0,$ we have $\|x-z\|\geq 1,$ a contradiction. Therefore, $H_0^\perp \cap H=\{\theta\}.$\\
	Now, if possible, suppose that $H^\perp \cap H_0\neq\{\theta\}.$ Then there exists $z\in S_{\mathbb{H}_1}$ such that $z\in H^\perp \cap H_0.$ Now, $z\in S_{H_0}$ implies that $z\in M_T,$ i.e., $\|Tz\|>1-\delta.$ Since $A_\epsilon$ is uniform $\epsilon-$BPB approximation, there exists $x\in M_{A_\epsilon}$ such that $\|z-x\|<\epsilon.$ Moreover, $z\in H^\perp$ and $x\in H$ gives that $\|z-x\|>1>\epsilon,$ a contradiction. Therefore, $H^\perp \cap H_0=\{\theta\}.$ This proves condition (i).\\
	
	(ii) Now, suppose that there exist $\epsilon_1,\epsilon_2>0$ with $\epsilon_1^2+\epsilon_2^2>2\epsilon^2.$ If possible, let $\|(T-A_\epsilon)|_{H_0}\|\geq\epsilon_1$ and  $\|(T-A_\epsilon)|_{H_0^\perp}\|\geq\epsilon_2.$ Then 
	$$2\epsilon^2>2\|T-A_\epsilon\|^2\geq \|(T-A_\epsilon)|_{H_0}\|^2+\|(T-A_\epsilon)|_{H_0^\perp}\|^2\geq \epsilon_1^2+\epsilon_2^2,$$
	a contradiction. Therefore, either $\|(T-A_\epsilon)|_{H_0}\|<\epsilon_1$ or $\|(T-A_\epsilon)|_{H_0^\perp}\|<\epsilon_2.$ This proves condition (ii).\\
	
	(iii) Next, let $z\in M_T(\delta).$ Then $\|Tz\|>1-\delta.$ Since $A_\epsilon$ is uniform $\epsilon-$BPB approximation of $T,$ there exists $x\in M_{A_\epsilon}$ such that $\|z-x\|<\epsilon,$ i.e., $z\in B(x,\epsilon).$ Thus, $M_T(\delta)\subseteq \cup_{x\in M_{A_\epsilon}}B(x,\epsilon).$ This proves condition (iii) and completes the proof of the theorem.
\end{proof}

As the final result of this article, we illustrate that when the dimension of the domain space is finite, the above theorem can be considerably improved.

\begin{theorem}
	Let $\mathbb{H}_1,\mathbb{H}_2$ be  Hilbert spaces such that $dim(\mathbb{H}_1)<\infty.$ Let $T\in S_{\mathbb{L}(\mathbb{H}_1,\mathbb{H}_2)}.$ Then there exists $\epsilon_0>0$ such that for any $0<\epsilon<\epsilon_0,$ if $A_\epsilon$ is a uniform $\epsilon-$BPB approximation of $T,$ then the following conditions hold true:\\
	(i) $dim(H_0)=dim(H),$ where $M_T=S_{H_0}$ and $M_{A_\epsilon}=S_{H}$ for some subspace $H_0,H$ of $\mathbb{H}_1,$\\
	(ii) for any $\epsilon_1,\epsilon_2>0$ with $\epsilon_1^2+\epsilon_2^2>2\epsilon^2$ either $\|(T-A_\epsilon)|_{H_0}\|<\epsilon_1$ or $\|(T-A_\epsilon)|_{H_0^\perp}\|<\epsilon_2,$\\
	(iii) there exists $\delta=\delta(\epsilon)>0$ such that $M_T(\delta)\subseteq \cup_{x\in M_{A_\epsilon}}B(x,\epsilon).$ 
\end{theorem}
\begin{proof}
	Since $dim(\mathbb{H}_1)< \infty,~M_T\neq \emptyset.$ Let $C$ be any closed subset of $S_{\mathbb{H}_1}$ with $d(M_T,C)>0.$ Then $C$ is a compact set. It is easy to observe that $\sup\{\|Tz\|:z\in C\}<1.$  Thus, the hypothesis of Theorem \ref{th-chhil}, is satisfied. Therefore, using Theorem \ref{th-chhil} we say that there exists $\epsilon_0>0$ such that for any $0<\epsilon<\epsilon_0,$ if $A_\epsilon$ is a uniform $\epsilon-$BPB approximation of $T,$ then conditions (ii) and (iii) of Theorem \ref{th-chhil} are true. We only show that condition (i) is true. Let  $M_T=S_{H_0}$ and $M_{A_\epsilon}=S_{H}$ for some subspace $H_0,H$ of $\mathbb{H}_1.$ Then by Theorem \ref{th-chhil}, we have  $H_0^\perp \cap H=\{\theta\}$ and $H^\perp\cap H_0=\{\theta\}.$ Let $dim(\mathbb{H}_1)=n,dim(H_0)=m$ and $dim(H)=k.$ Then 
	\begin{eqnarray*}
		0=dim(H_0^\perp\cap H)&=&dim(H_0^\perp)+dim(H)-dim(H_0^\perp +H)\\
		&=& n-m+k+dim(H_0^\perp+H)\\
		&\geq& n-m+k-n=k-m.
	\end{eqnarray*}
	Thus, $m\geq k.$ Similarly, $H^\perp\cap H_0=\{\theta\}$ gives that $k\geq m.$ Therefore, $m=k,$ i.e., $dim(H_0)=dim(H).$ This proves condition (i) and completes the proof of the theorem. 
\end{proof}

\bibliographystyle{amsplain}

\end{document}